\documentclass[reqno,a4paper,11pt]{amsart}

\usepackage[utf8]{inputenc}
\usepackage{graphicx}
\usepackage{subfigure}
\usepackage{amsmath, amssymb, amsthm, mathtools, braket,times}
\usepackage{amsaddr}
\usepackage{enumitem}
\usepackage{cite}
\usepackage{graphicx}
\usepackage{array}
\usepackage{tikz}
\usepackage{mathrsfs}
\usetikzlibrary{arrows}
\usetikzlibrary{shapes}
\usetikzlibrary{decorations.pathreplacing}
\usepackage{cleveref}
\usepackage[english]{babel}
\usepackage[T1]{fontenc}
\usepackage{fourier}
\usepackage{bbm}
\usepackage{color}
\usepackage{fullpage}

\newcommand{\whp}{whp}
\newcommand{\prob}[1]{\mathbb{P}\left[#1\right]}    
\newcommand{\condprob}[2]{\mathbb{P}\left[#1 \;\middle|\; #2\right]}

\newcommand{\expec}[1]{\mathbb{E}\left[#1\right]}    
\DeclareMathOperator{\po}{Po}
\newcommand{\poisson}[1]{\po\left(#1\right)}

\def\Erdos{Erd\H{o}s}
\def\Renyi{R\'enyi}
\def\Luczak{\L{}uczak}
\def\ER{\Erdos--\Renyi}

\def\GW{Galton--Watson}

\newtheorem{thm}{Theorem}[section]

\newtheorem{coro}[thm]{Corollary}
\newtheorem{lem}[thm]{Lemma}

\theoremstyle{remark}

\theoremstyle{definition}
\newtheorem{definition}[thm]{Definition}

\newtheorem*{construction*}{Construction}
\newcommand{\proofof}[1]{\subsection{Proof of \Cref{#1}}}
\newcommand{\proofofW}[1]{\subsection*{Proof of \Cref{#1}}}
\crefname{thm}{theorem}{theorems}
\crefname{prop}{proposition}{propositions}
\crefname{coro}{corollary}{corollaries}
\crefname{lem}{lemma}{lemmas}
\crefname{definition}{definition}{definitions}
\crefname{question}{question}{questions}
\crefname{problem}{problem}{problems}
\crefname{conjecture}{conjecture}{conjectures}

\newtheoremstyle{claim}
{}
{}
{\itshape}
{}
{\bf}
{.}
{.5em}
{}
\theoremstyle{claim}

\crefname{claim}{claim}{claims}

\def\N{\mathbb{N}} 
\def\ur{\in_R} 

\def\forest{F}
\def\forestClass{\mathcal{F}}

\newcommand{\smallo}[1]{o\left(#1\right)}
\newcommand{\bigo}[1]{O\left(#1\right)}
\newcommand{\smallomega}[1]{\omega\left(#1\right)}

\def\planargraph{P}     
\def\planarclass{\mathcal{P}} 
\def\erClass{\mathcal{G}} 
\def\nocomplexClass{\mathcal{U}} 
\newcommand{\largestcomponent}[1]{L\left(#1\right)} 
\def\Largestcomponent{L} 
\newcommand{\rest}[1]{S\left(#1\right)} 
\def\Rest{S} 
\newcommand{\vertexSet}[1]{V\left(#1\right)} 
\newcommand{\edgeSet}[1]{E\left(#1\right)} 
\newcommand{\numberVertices}[1]{v\left({#1}\right)} 
\newcommand{\numberEdges}[1]{e\left({#1}\right)} 
\DeclareMathOperator{\dist}{dist}
\newcommand{\distance}[3]{\dist_{#1}\left(#2, #3\right)}
\newcommand{\degree}[2]{d_{#2}\left(#1\right)} 

\newcommand{\core}[1]{C\left(#1\right)} 
\newcommand{\kernel}[1]{K\left(#1\right)} 
\newcommand{\complexpart}[1]{Q\left(#1\right)} 
\newcommand{\restcomplex}[1]{U\left(#1\right)} 

\def\Restcomplex{U} 

\def\numberVerticesOutside{n_U}
\def\numberEdgesOutside{m_U}
\newcommand{\subdivisionNumber}[2]{S_{#1}\left(#2\right)}
\def\cl{\mathcal{A}} 
\def\func{\Phi} 
\def\seq{\mathbf{a}} 
\def\term{a} 
\newcommand{\condGraph}[2]{#1 \mid #2} 
\def\imageSet{\mathcal{S}} 

\def\root{r} 
\def\rootSet{R} 
\def\radius{\ell} 
\newcommand{\ball}[3]{B_{#1}\left(#2, #3\right)}
\newcommand{\ballP}[2]{B_{#1}\left(#2\right)}
\def\isomorphic{\cong}
\newcommand{\rootedGraph}[2]{\left(#1, #2\right)}
\newcommand{\distributionalLimit}[2]{#1~\xrightarrow{~d~}~#2}
\DeclareMathOperator{\GWT}{GWT}
\newcommand{\gwt}[1]{\GWT\left(#1\right)}
\def\skeletonTree{T_\infty}
\newcommand{\skeletonTreeRays}[1]{T_\infty^{\left(#1\right)}}
\newcommand{\unbiased}[1]{#1_\circ}
\newcommand{\contiguous}[2]{#1 \triangleleft \triangleright #2}
\def\detGraph{W} 
\def\planeTree{Z} 
\newcommand{\infinitepath}[1]{P_{\infty}^{\left(#1\right)}}
\newcommand{\setbuilder}[2]{\left\{#1 \mid #2\right\}} 
\newcommand{\setbuilderBig}[2]{\big\{#1 \mid #2\big\}}

\newcommand{\symmetricDifference}[2]{#1\triangle#2}
\def\ntrees{t} 

\title{Local limit of sparse random planar graphs}
\author{Mihyun Kang, Michael Missethan}
\address{Institute of Discrete Mathematics, Graz University of Technology, Steyrergasse 30, 8010 Graz, Austria}
\email{\{kang,missethan\}@math.tugraz.at}
\thanks{Supported by Austrian Science Fund (FWF): W1230}

\begin{document}

\begin{abstract}
Let $\planargraph(n,m)$ be a graph chosen uniformly at random from the class of all planar graphs on vertex set $\left\{1, \ldots, n\right\}$ with $m=m(n)$ edges. We determine the (Benjamini--Schramm) local weak limit of $\planargraph(n,m)$ in the sparse regime when $m\leq n+\smallo{n\left(\log n\right)^{-2/3}}$. Assuming that the average degree $2m/n$ tends to a constant $c\in[0,2]$ the local weak limit of $\planargraph(n,m)$ is a \GW\ tree with offspring distribution $\poisson{c}$ if $c\leq 1$, while it is the \lq Skeleton tree\rq\ if $c=2$. Furthermore, there is a \lq smooth\rq\ transition between these two cases in the sense that the local weak limit of $\planargraph(n,m)$ is a \lq linear combination\rq\ of a \GW\ tree and the Skeleton tree if $c\in\left(1,2\right)$.
\end{abstract}

\maketitle

\section{Introduction and results}\label{sec:intro}

\subsection{Motivation}\label{sub:motivation}
In 1959 \Erdos\ and \Renyi\ \cite{erdoes1} introduced the so--called \ER\ random graph $G(n,m)$, which is a graph chosen uniformly at random from the class $\mathcal{G}(n,m)$ of all vertex--labelled graphs on vertex set $[n]:=\left\{1, \ldots, n\right\}$ with $m=m(n)$ edges, denoted by $G(n,m)\ur\mathcal{G}(n,m)$. Since then, $G(n,m)$ and the closely related binomial random graph $G(n,p)$ have been intensively studied (see \cite{rg1,rg2,rg3} for an overview). Prominent examples of properties that were considered in $G(n,m)$ are the order of the largest component \cite{erdoes2,luczak_component,bollobas_component}, the length of the longest cycle \cite{luczak_cycles,luczak_pittel_wierman,longCycle}, and the maximum degree \cite{vertices_given_degree,max_degree1,max_degree2}. While these are all \lq global\rq\ properties, also the \lq local\rq\ structure of $G(n,m)$ was studied. For instance, it is well known that if $m=cn/2$ for a constant $c\geq 0$, then $G(n,m)$ locally looks like a \GW\ tree $\gwt{c}$ with offspring distribution $\poisson{c}$. To make that statement more formal, we use the concept of {\em local weak convergence} of random rooted graphs, which was introduced by Benjamini and Schramm \cite{benjamini_schramm} and by Aldous and Steele \cite{objective_method}: A {\em rooted graph} is a pair $\rootedGraph{H}{\root}$, where $H$ is a graph and $\root$ is a vertex of $H$. Two rooted graphs $\rootedGraph{H_1}{\root_1}$ and $\rootedGraph{H_2}{\root_2}$ are {\em isomorphic}, denoted by $\rootedGraph{H_1}{\root_1}\isomorphic\rootedGraph{H_2}{\root_2}$, if there exists an isomorphism of $H_1$ onto $H_2$ which takes $r_1$ to $r_2$. Given a rooted graph $\rootedGraph{H}{\root}$ and $\radius\in\N:=\left\{1, 2, \ldots\right\}$, the ball $\ball{\radius}{H}{\root}$ of radius $\radius$ and centre $\root$ in $H$ is the subgraph of $H$ induced on the vertices with distance at most $\radius$ to $\root$. For each $n\in\N$ let $\rootedGraph{G}{\root}=\rootedGraph{G}{\root}(n)$ be a random rooted graph, that is a random variable which takes rooted graphs as values. We say that another random rooted graph $\rootedGraph{G_0}{\root_0}$ is the {\em local weak limit}, or just the local limit, of $\rootedGraph{G}{\root}$, denoted by $\distributionalLimit{\rootedGraph{G}{\root}}{\rootedGraph{G_0}{\root_0}}$, if for each fixed rooted graph $\rootedGraph{H}{\root_H}$ and $\radius\in\N$ we have
\begin{align*}
\prob{\ball{\radius}{G}{\root}\isomorphic\rootedGraph{H}{\root_H}}~\to~ \prob{\ball{\radius}{G_0}{\root_0}\isomorphic\rootedGraph{H}{\root_H}} \text{~~~~as~~~~} n\to \infty.
\end{align*}
In the literature the local weak limit is also called the distributional limit or Benjamini--Schramm local weak limit. We can consider the \ER\ random graph $G=G(n,m)$ as a random rooted graph $\rootedGraph{G}{\root}$ by assuming that given the realisation $H$ of $G$ the root $\root$ is chosen uniformly at random from the vertices of $H$, denoted by $r\ur\vertexSet{G}$. Using the notions introduced above we can describe the local structure of $G$ as follows.
\begin{thm}\label{thm:local_er}
Let $G=G(n,m)\ur\mathcal{G}(n,m)$ be the \ER\ random graph, $\root=\root(n)\ur \vertexSet{G}$, $c\geq 0$, and $m=m(n)$ such that $2m/n\to c$ as $n \to \infty$. Then $\distributionalLimit{\rootedGraph{G}{\root}}{\gwt{c}}$.
\end{thm}
A proof of \Cref{thm:local_er} can be found for example in \cite{local_limit_er_proof}. Another random graph for which the local structure is well understood is the random tree $T=T(n)$, which is a graph chosen uniformly at random from the class of all vertex--labelled trees on vertex set $[n]$. Grimmett \cite{limit_random_tree} showed that the local limit of $T$ is the {\em Skeleton tree} $\skeletonTree$, which can be obtained by connecting independent \GW\ trees $T_1, T_2, \ldots$ with offspring distribution $\poisson{1}$ and roots $\root_1, \root_2, \ldots$ via edges $\root_1\root_2, \root_2\root_3, \ldots$ and considering $\root_1$ as the root of $\skeletonTree$. We refer to \Cref{fig:skeleton_tree} for an illustration of $\skeletonTree$ and note that further results on limits of random trees can be found for example in \cite{tree_limit1,tree_limit2,continuum_random_tree}.
\begin{figure}[t]
	\begin{tikzpicture}[scale=1, line width=0.4mm, every node/.style={circle,fill=gray, inner sep=0, minimum size=0.22cm}]
		\node (1) at (0,0) [rectangle,draw]  {};
		\node (1a) at (0,-0.4) [draw=none, fill=none]  {$\root_1$};
		\node (2) at (1,0) [draw]  {};
		\node (2a) at (1,-0.4) [draw=none, fill=none]  {$\root_2$};
		\node (3) at (2,0) [draw]  {};
		\node (3a) at (2,-0.4) [draw=none, fill=none]  {$\root_3$};
		\node (4) at (3,0) [draw]  {};
		\node (4a) at (3,-0.4) [draw=none, fill=none]  {$\root_4$};
		\node (5) at (4,0) [draw]  {};
		\node (5a) at (4,-0.4) [draw=none, fill=none]  {$\root_5$};
		\node (6) at (5,0) [draw]  {};
		\node (6a) at (5,-0.4) [draw=none, fill=none]  {$\root_6$};
		\node (7) at (6,0) [draw]  {};
		\node (7a) at (6,-0.4) [draw=none, fill=none]  {$\root_7$};
		\node (8) at (7,0) [draw]  {};
		\node (8a) at (7,-0.4) [draw=none, fill=none]  {$\root_8$};
		\draw[-] (1) to (2);
		\draw[-] (2) to (3);
		\draw[-] (3) to (4);
		\draw[-] (4) to (5);
		\draw[-] (5) to (6);
		\draw[-] (6) to (7);
		\draw[-] (7) to (8);
		
		\node (9) at (0,0.7) [draw]  {};
		\draw[-] (1) to (9);
		\node (10) at (0.7,0.7) [draw]  {};
		\node (11) at (1.3,0.7) [draw]  {};
		\draw[-] (2) to (10);
		\draw[-] (2) to (11);
		\node (12) at (0.2,1.4) [draw]  {};
		\node (13) at (0.7,1.4) [draw]  {};
		\node (14) at (1.2,1.4) [draw]  {};
		\draw[-] (10) to (12);
		\draw[-] (10) to (13);
		\draw[-] (10) to (14);
		\node (15) at (0.2,2.1) [draw]  {};
		\node (16) at (0.9,2.1) [draw]  {};
		\node (17) at (1.5,2.1) [draw]  {};
		\draw[-] (12) to (15);
		\draw[-] (14) to (16);
		\draw[-] (14) to (17);
		\node (18) at (-0.1,2.8) [draw]  {};
		\node (19) at (0.5,2.8) [draw]  {};	
		\draw[-] (15) to (18);
		\draw[-] (15) to (19);
		\node (20) at (0.5,3.5) [draw]  {};
		\draw[-] (19) to (20);	
		
		\node (21) at (5.6,0.7) [draw]  {};
		\node (22) at (6,0.7) [draw]  {};
		\node (23) at (6.4,0.7) [draw]  {};
		\draw[-] (7) to (21);	
		\draw[-] (7) to (22);	
		\draw[-] (7) to (23);
		\node (24) at (5.3,1.4) [draw]  {};
		\node (25) at (5.9,1.4) [draw]  {};
		\node (26) at (6.4,1.4) [draw]  {};
		\draw[-] (21) to (24);	
		\draw[-] (21) to (25);	
		\draw[-] (23) to (26);
		\node (27) at (5.3,2.1) [draw]  {};
		\node (28) at (6,2.1) [draw]  {};
		\node (29) at (6.4,2.1) [draw]  {};
		\node (30) at (6.8,2.1) [draw]  {};
		\draw[-] (24) to (27);	
		\draw[-] (26) to (28);	
		\draw[-] (26) to (29);
		\draw[-] (26) to (30);	
		\node (31) at (5,2.8) [draw]  {};
		\node (32) at (5.6,2.8) [draw]  {};
		\node (33) at (6.4,2.8) [draw]  {};
		\node (34) at (6.8,2.8) [draw]  {};
		\draw[-] (27) to (31);	
		\draw[-] (27) to (32);	
		\draw[-] (29) to (33);
		\draw[-] (30) to (34);
		\node (35) at (5,3.5) [draw]  {};
		\node (36) at (6.2,3.5) [draw]  {};
		\node (37) at (6.6,3.5) [draw]  {};
		\draw[-] (31) to (35);	
		\draw[-] (33) to (36);	
		\draw[-] (33) to (37);
		\node (38) at (5,4.2) [draw]  {};
		\node (39) at (6.2,4.2) [draw]  {};
		\draw[-] (35) to (38);	
		\draw[-] (36) to (39);	
		
		\node (40) at (4,0.7) [draw]  {};
		\node (41) at (3.7,1.4) [draw]  {};
		\node (42) at (4.3,1.4) [draw]  {};
		\node (43) at (3.7,2.1) [draw]  {};	
		\draw[-] (5) to (40);	
		\draw[-] (40) to (41);	
		\draw[-] (40) to (42);	
		\draw[-] (41) to (43);	
		
		\node (44) at (3,0.7) [draw]  {};
		\node (45) at (3,1.4) [draw]  {};	
		\draw[-] (4) to (44);	
		\draw[-] (44) to (45);	
		
		\node (46) at (7.8,0) [draw=none, fill=none]  {};
		\draw[->] (8) to (46);
		\node (47) at (8.2,1) [draw=none, fill=none]  {\Huge $\ldots$};
	\end{tikzpicture}
	\caption{The Skeleton Tree $\skeletonTree$ with root $\root=\root_1$.}
	\label{fig:skeleton_tree}
\end{figure}
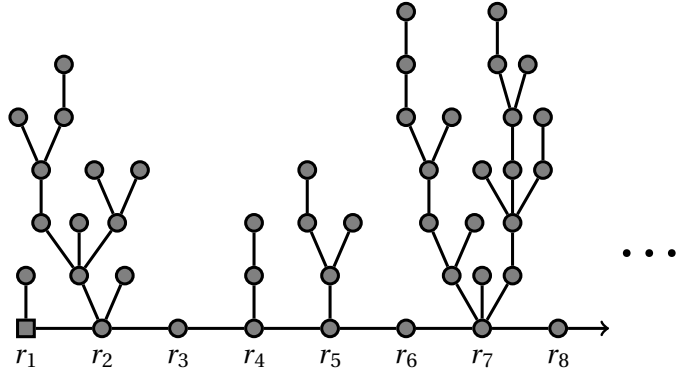
\begin{thm}[\cite{limit_random_tree}]\label{thm:limit_tree1}
Let $T=T(n)$ be a random tree and $\root=\root(n)\ur\vertexSet{T}$. Then $\distributionalLimit{\rootedGraph{T}{\root}}{\skeletonTree}$.
\end{thm}
In the last decades random planar graphs and, more generally, random graphs on surfaces have attracted attention \cite{planar,planar1,planar2,planar3,planar4,planar5,mcd,surface,surface1}. In this paper we investigate the local structure of $\planargraph(n,m)$, which is a graph chosen uniformly at random from the class $\planarclass(n,m)$ of all vertex--labelled planar graphs on vertex set $[n]$ with $m=m(n)$ edges, denoted by $\planargraph(n,m)\ur\planarclass(n,m)$. In particular, we determine the local weak limit of $\planargraph(n,m)$, in the flavour of \Cref{thm:local_er}. We note that there are results on similar types of convergence in related structures known, like random planar maps \cite{local_maps}, random outerplanar graphs \cite{local_outerplanar}, and $n$--vertex random planar graphs \cite{local_n_vertex}.

\subsection{Main results}\label{sub:main}
Throughout the paper, we use the standard Landau notations for asymptotic orders and all asymptotics are taken as $n\to\infty$. We investigate the local weak limit of $\planargraph(n,m)$ in the sparse regime when $m\leq n+\smallo{n\left(\log n\right)^{-2/3}}$ and distinguish four cases depending on how large the number of edges $m$ is. The first regime is when $m\leq n/2+\bigo{n^{2/3}}$, where $\planargraph(n,m)$ behaves very similarly as the \ER\ random graph $G(n,m)$ and has therefore also a \GW\ tree as its local limit.
\begin{thm}\label{thm:main1}
Let $\planargraph=\planargraph(n,m)\ur\planarclass(n,m)$, $\root=\root(n)\ur \vertexSet{\planargraph}$, and $m=m(n)\leq n/2+\bigo{n^{2/3}}$ such that $2m/n \to c\in[0,1]$. Then $\distributionalLimit{\rootedGraph{\planargraph}{\root}}{\gwt{c}}$.	
\end{thm}

Kang and \Luczak\ \cite{planar} showed that if $m\geq n/2+\smallomega{n^{2/3}}$, then with high probability (meaning with probability tending to 1 as $n$ tends to infinity, {\em \whp} for short) the largest component of $\planargraph=\planargraph(n,m)$ has significantly more vertices than the second largest component. Therefore, we determine in the following the local limit not only for $\planargraph$, but also for the largest component $\largestcomponent{\planargraph}$ of $\planargraph$ and the remaining part $\rest{\planargraph}:=\planargraph\setminus\largestcomponent{\planargraph}$, which we call (the union of) the small components. We write $\root_L\ur\vertexSet{\Largestcomponent}$ (and $\root_\Rest\ur\vertexSet{\Rest}$) for a vertex that is obtained by first picking the realisation $H$ of $\planargraph$ and then choosing uniformly at random a vertex from the largest component $\largestcomponent{H}$ of $H$ (and the small components $\rest{H}$, respectively). The next case is when $m\geq n/2+\smallomega{n^{2/3}}$, but the average degree $2m/n$ still tends to 1.
\begin{thm}\label{thm:main2}
Let $\planargraph=\planargraph(n,m)\ur\planarclass(n,m)$, $\Largestcomponent=\largestcomponent{P}$ be the largest component of $P$, and $\Rest=\rest{\planargraph}=\planargraph\setminus\Largestcomponent$. Assume $m=m(n)=n/2+s$ for $s=s(n)>0$ such that $s=\smallo{n}$ and $s^3n^{-2}\to\infty$. Moreover, let $\root=\root(n)\ur\vertexSet{\planargraph}$, $\root_L=\root_L(n)\ur\vertexSet{\Largestcomponent}$, and $\root_\Rest=\root_\Rest(n)\ur\vertexSet{\Rest}$. Then we have
\begin{enumerate}[label=\normalfont(\roman*)]
\item\label{thm:main2A} $\distributionalLimit{\rootedGraph{\Largestcomponent}{\root_L}}{\skeletonTree}$;
\item\label{thm:main2B}
$\distributionalLimit{\rootedGraph{\Rest}{\root_\Rest}}{\gwt{1}}$;
\item\label{thm:main2C}
$\distributionalLimit{\rootedGraph{\planargraph}{\root}}{\gwt{1}}$.
\end{enumerate}
\end{thm}
Next, we consider the regime when the average degree $2m/n$ tends to a constant $c\in (1,2)$. Then the local limit of $\planargraph(n,m)$ is a \lq linear combination\rq\ of the \GW\ tree $\gwt{1}$ and the Skeleton tree $\skeletonTree$. More formally, given two random rooted graphs $\rootedGraph{G_1}{\root_1}, \rootedGraph{G_2}{\root_2}$ and a constant $a\in[0,1]$, we write $a\rootedGraph{G_1}{\root_1}+(1-a)\rootedGraph{G_2}{\root_2}$ for the random rooted graph $\rootedGraph{G}{\root}$ that satisfies 
\begin{align*}
\prob{\rootedGraph{G}{\root}=\rootedGraph{H}{\root_H}}=a~\prob{\rootedGraph{G_1}{\root_1}=\rootedGraph{H}{\root_H}}+(1-a)~\prob{\rootedGraph{G_2}{\root_2}=\rootedGraph{H}{\root_H}},
\end{align*}
for each fixed rooted graph $\rootedGraph{H}{\root_H}$.
\begin{thm}\label{thm:main3}
	Let $\planargraph=\planargraph(n,m)\ur\planarclass(n,m)$, $\Largestcomponent=\largestcomponent{P}$ be the largest component of $P$, and $\Rest=\rest{\planargraph}=\planargraph\setminus\Largestcomponent$. Assume $m=m(n)=\alpha n/2$, where $\alpha=\alpha(n)$ tends to a constant $c\in (1,2)$. Moreover, let $\root=\root(n)\ur\vertexSet{\planargraph}$, $\root_L=\root_L(n)\ur\vertexSet{\Largestcomponent}$, and $\root_\Rest=\root_\Rest(n)\ur\vertexSet{\Rest}$. Then we have
	\begin{enumerate}[label=\normalfont(\roman*)]
		\item\label{thm:main3A} $\distributionalLimit{\rootedGraph{\Largestcomponent}{\root_L}}{\skeletonTree}$;
		\item\label{thm:main3B}
		$\distributionalLimit{\rootedGraph{\Rest}{\root_\Rest}}{\gwt{1}}$;
		\item\label{thm:main3C}
		$\distributionalLimit{\rootedGraph{\planargraph}{\root}}{\left(c-1\right)\skeletonTree+\left(2-c\right)\gwt{1}}$.
	\end{enumerate}
\end{thm}

Finally, we consider the case that the average degree $2m/n$ tends to 2 and $m\leq n+\smallo{n\left(\log n\right)^{-2/3}}$.
\begin{thm}\label{thm:main4}
	Let $\planargraph=\planargraph(n,m)\ur\planarclass(n,m)$, $\Largestcomponent=\largestcomponent{P}$ be the largest component of $P$, and $\Rest=\rest{\planargraph}=\planargraph\setminus\Largestcomponent$. Assume $m=m(n)$ is such that $m=n+\smallo{n}$ and $m\leq n+\smallo{n\left(\log n\right)^{-2/3}}$. Moreover, let $\root=\root(n)\ur\vertexSet{\planargraph}$, $\root_L=\root_L(n)\ur\vertexSet{\Largestcomponent}$, and $\root_\Rest=\root_\Rest(n)\ur\vertexSet{\Rest}$. Then we have
	\begin{enumerate}[label=\normalfont(\roman*)]
		\item $\distributionalLimit{\rootedGraph{\Largestcomponent}{\root_L}}{\skeletonTree}$;
		\item
		$\distributionalLimit{\rootedGraph{\Rest}{\root_\Rest}}{\gwt{1}}$;
		\item
		$\distributionalLimit{\rootedGraph{\planargraph}{\root}}{\skeletonTree}$.
	\end{enumerate}
\end{thm}
We conclude this section with some observations on the case that $m$ is as in \Cref{thm:main2}, \ref{thm:main3}, or \ref{thm:main4}. The local limit of the largest component $\largestcomponent{\planargraph}$ of $\planargraph=\planargraph(n,m)$ is the Skeleton tree $\skeletonTree$ and that of the small components $\rest{\planargraph}$ is the \GW\ tree $\gwt{1}$. In view of \Cref{thm:local_er,thm:limit_tree1} this indicates that $\largestcomponent{\planargraph}$ locally looks like a tree and $\rest{\planargraph}$ like the \ER\ random graph with average degree 1. Moreover, \Cref{thm:main3} provides a {\em smooth} transition between the cases when the average degree $2m/n$ is tending to 1 and 2, in the sense that we obtain the statements of \Cref{thm:main2,thm:main4} if we plug in the extreme cases $c=1$ and $c=2$ in \Cref{thm:main3}, respectively. \Cref{thm:main3}\ref{thm:main3C} says that the local limit of $\planargraph$ is a \lq linear combination\rq\ of the local limit of the largest component $\largestcomponent{\planargraph}$ and that of the small components $\rest{\planargraph}$. 

\subsection{Outline of the paper} 
The rest of the paper is structured as follows. After setting the basic notations and definitions in \Cref{sec:prelim}, we present our proof strategy in \Cref{sec:strategy}. In \Cref{sec:local_er_non_complex} we provide statements on the local structure of the \ER\ random graph and the so--called non--complex part, which will be later the main ingredients to determine the local limit of the small components. Similarly, we consider in \Cref{sec:local_complex} the so--called complex part, which we will use later to deduce the local limit of the largest component. \Cref{sec:proofs} is devoted to the proofs of the main results. Subsequently in \Cref{sec:local_core_kernel}, we look deeper into the local structure of a random planar graph by considering cores and kernels. In \Cref{sec:discussion}, we conclude with a possible generalisation of our results.

\section{Preliminaries}\label{sec:prelim}
\subsection{Notations for graphs}
Throughout the paper, we consider only undirected (simple or multi) graphs.
\begin{definition}
	Given a (simple or multi) graph $H$ we denote by
	\begin{itemize}
		\item 
		$\vertexSet{H}$ the vertex set of $H$ and
		\item[]
		$\numberVertices{H}$ the order of $H$, i.e. the number of vertices in $H$;
		\item 
		$\edgeSet{H}$ the edge set of $H$ and
		\item[]
		$\numberEdges{H}$ the size of $H$, i.e. the number of edges in $H$;	
		\item 
		$\largestcomponent{H}$ the largest component of $H$;
		\item
		$\rest{H}:=H \setminus \largestcomponent{H}$ the graph obtained from $H$ by deleting the largest component, called (the union of) the small components;
		\item
		$\distance{H}{v}{w}$ the distance between the vertices $v,w\in\vertexSet{H}$ in $H$, i.e. the length of the shortest path from $v$ to $w$. 
	\end{itemize}
\end{definition}
\begin{definition}\label{def:graph_class}
	Given a class $\cl$ of graphs, we write
	$\cl(n)$ for the subclass of $\cl$ containing the graphs on vertex set $[n]$ 
	and $\cl(n,m)$ for the subclass of $\cl$ containing the graphs on vertex set $[n]$ with $m$ edges, respectively. We denote by $A(n)\ur \cl(n)$ a graph chosen uniformly at random from $\cl(n)$
	and by $A(n,m)\ur \cl(n,m)$ a graph chosen uniformly at random from $\cl(n,m)$, respectively. We tacitly assume that $|\cl(n)|$ and $|\cl(n,m)|$ are finite for all considered classes $\cl$ and all $n,m\in \N$.  
\end{definition}

\subsection{Complex part, core, and kernel}\label{sub:decomposition}
Given a graph $H$, we call a component {\em complex} if it has at least two cycles. Furthermore, we say that $H$ is {\em complex} if all its components are. We denote by $\complexpart{H}$ the {\em complex part} of $H$, which is the union of all complex components. The remaining part $\restcomplex{H}:=H\setminus\complexpart{H}$ is called the {\em non--complex part} of $H$. We denote by $n_U(H):=\numberVertices{\restcomplex{H}}$ and $m_U(H):=\numberEdges{\restcomplex{H}}$ the number of vertices and edges in $\restcomplex{H}$, respectively. The {\em core} $\core{H}$ is the maximal subgraph of $\complexpart{H}$ with minimum degree at least two. Equivalently, the core $\core{H}$ is obtained by recursively deleting vertices of degree one in $\complexpart{H}$. Conversely, the complex part $\complexpart{H}$ can be constructed by replacing each vertex in $\core{H}$ by a rooted tree. Finally, we obtain the {\em kernel} $\kernel{H}$ of $H$ by replacing all maximal paths in $\core{H}$ having only internal vertices of degree two by an edge between the endpoints of the path. We note that the kernel $\kernel{H}$ can contain multiple edges and loops. Reversing this construction, the core $\core{H}$ arises from the kernel $\kernel{H}$ by subdividing edges.

Later we will see that the local structures of the largest component $\largestcomponent{\planargraph}$ of $\planargraph=\planargraph(n,m)$ and the small components $\rest{\planargraph}$ are closely related to the local limits of the complex part $\complexpart{\planargraph}$ and the non--complex part $\restcomplex{\planargraph}$, respectively. In order to analyse $\complexpart{\planargraph}$ and $\restcomplex{\planargraph}$, we will use the following random graphs.

\begin{definition}\label{def:random_complex}
	Let $C$ be a core, i.e. a graph with minimum degree at least two, and $q\in\N$. We denote by $Q(C,q)$ a graph chosen uniformly at random from the class of all complex graphs having core $C$ and vertex set $[q]$.
\end{definition}

\begin{definition}\label{def:random_non_complex}
	We let $\nocomplexClass$ be the class of all graphs without complex components. For $n,m\in\N$ we denote by $\nocomplexClass(n,m)$ the subclass of $\nocomplexClass$ containing all graphs on vertex set $[n]$ with $m$ edges and write $U(n,m)\ur \nocomplexClass(n,m)$ for a graph chosen uniformly at random from $\nocomplexClass(n,m)$.
\end{definition}

\subsection{Internal structure of a random planar graph}
In the proofs of \Cref{thm:main2,thm:main3,thm:main4} we will use results on the internal structure of $\planargraph=\planargraph(n,m)$ from \cite{surface}, e.g. the order of the largest component $\largestcomponent{\planargraph}$, the complex part $\complexpart{\planargraph}$, and the core $\core{\planargraph}$.
\begin{thm}[\cite{surface}]\label{thm:internal_structure}
	Let $\planargraph=\planargraph(n,m)\ur\planarclass(n,m)$ be the random planar graph, $Q=\complexpart{\planargraph}$ the complex part of $\planargraph$, $C=\core{\planargraph}$ the core, $K=\kernel{\planargraph}$ the kernel, $\largestcomponent{\planargraph}$ the largest component, and $\rest{\planargraph}$ the small components. We denote by $\largestcomponent{Q}$ and $\rest{Q}$ the largest component and the small components of $Q$, respectively. Moreover, let $\numberVerticesOutside=\numberVerticesOutside(\planargraph)$ and $\numberEdgesOutside=\numberEdgesOutside(\planargraph)$ be the number of vertices and edges in the non--complex part $\restcomplex{\planargraph}$ of $\planargraph$, respectively. We assume that either $m=n/2+s$ for $s=s(n)>0$ such that $s=\smallo{n}$ and $s^3n^{-2}\to \infty$, $m=\alpha n/2$ for $\alpha=\alpha(n)$ tending to a constant in $(1,2)$, or $m=n+\smallo{n}$ and $m\leq n+\smallo{n\left(\log n\right)^{-2/3}}$. Then \whp
	\begin{enumerate}[label=\normalfont(\roman*)]
		\item\label{thm:internal_structure1}
		$\numberEdges{K}=\smallo{\numberVertices{C}}$;
		\item\label{thm:internal_structure9}
		$\numberEdges{K}=\left(3/2+\smallo{1}\right)\numberVertices{K}$;
		\item\label{thm:internal_structure2}
		$\numberVertices{C}=\smallo{\numberVertices{Q}}$;
		\item\label{thm:internal_structure3}
		$\numberVertices{\largestcomponent{\planargraph}}=\left(2m/n-1+\smallo{1}\right)n$;
		\item\label{thm:internal_structure4}
		$\numberVerticesOutside=\smallomega{1}$;
		\item\label{thm:internal_structure5}
		$m_U= n_U/2+\bigo{hn_U^{2/3}}$ for each function $h=h(n)=\smallomega{1}$;
		\item\label{thm:internal_structure6}
		$\largestcomponent{\planargraph}=\largestcomponent{Q}$;
		\item\label{thm:internal_structure7}
		$\numberVertices{\rest{Q}}=\smallo{\numberVertices{Q}}$;
		\item\label{thm:internal_structure8}
		$\numberVertices{\rest{Q}}=\smallo{\numberVerticesOutside}$.
	\end{enumerate}
\end{thm}
\subsection{Rooted graphs}
We can relate each random graph $G$ to a random rooted graph $\unbiased{G}$ in a canonical way as follows.
\begin{definition}\label{def:unbiased}
Given a random graph $G$, the random rooted graph $\unbiased{G}$ is the random rooted graph $\rootedGraph{G}{\root}$ where we first sample $G$ and given the realisation $H$ of $G$ the root $\root\ur\vertexSet{H}$ is chosen uniformly at random from $\vertexSet{H}$.
\end{definition}
In \cite{benjamini_schramm} the random rooted graph $\unbiased{G}$ is also called {\em unbiased}, as the root of $\unbiased{G}$ is chosen {\em uniformly} at random. In some cases we want to pick the root only from a subset of vertices rather than from the whole vertex set as in \Cref{def:unbiased}, e.g. from the set of all vertices that are in the largest component. 

\begin{definition}\label{def:R_rooted}
Let $\cl$ be a class of graphs and for each graph $H\in\cl$ let $\rootSet=\rootSet(H)\subseteq\vertexSet{H}$ be a non--empty subset. The {\em random $\rootSet$--rooted graph} $A_{\rootSet}=A_{\rootSet}(n)$ is the random rooted graph $\rootedGraph{A}{\root_{\rootSet}}$ where
\begin{itemize}
\item
$A=A(n)\ur\cl(n)$ is chosen uniformly at random from $\cl(n)$ and
\item
given the realisation $H$ of $A$, $\root_{\rootSet}\ur\rootSet(H)$ is chosen uniformly at random from $\rootSet(H)$.
\end{itemize}
\end{definition}
In most of our applications of \Cref{def:R_rooted} the subset $\rootSet=\rootSet(H)$ will be given as the vertex set $V(H')$ of a subgraph $H'$ of $H$, i.e. $\rootSet=\vertexSet{H'}$. With a slight abuse of notation we will write $A_{H'}$ instead of $A_{V(H')}$, e.g. $A_\Largestcomponent$ denotes the random rooted graph where the root is chosen uniformly at random from $V(\Largestcomponent)$, i.e. the set of all vertices that lie in the largest component $\Largestcomponent$. If we want to specify the root of the random rooted graph $A_{\rootSet}$ explicitly, we will use $\rootedGraph{A}{\root_{\rootSet}}$ instead of $A_{\rootSet}$. The next lemma will allow to analyse the local structure of a random rooted graph by means of random $\rootSet$--rooted graphs. In particular, we will use it to deduce the local limit of the random planar graph $\planargraph=\planargraph(n,m)$ from the local limits of the largest component $\largestcomponent{\planargraph}$ and the small components $\rest{\planargraph}$.

\begin{lem}\label{lem:linear_combination}
Let $\cl$ be a class of graphs and $A=A(n)\ur\cl(n)$. Furthermore, for each $H\in\cl$ let $\rootSet=\rootSet(H)\subseteq\vertexSet{H}$ and $\overline{\rootSet}=\overline{\rootSet}(H):=\vertexSet{H}\setminus\rootSet$ be non--empty subsets. We suppose that there exist a constant $a \in [0,1]$ and random rooted graphs $G_1$ and $G_2$ such that \begin{enumerate}[label=\normalfont(\roman*)]
\item\label{ass:1}
\whp\ $\left|\rootSet\left(A\right)\right|=\left(a+\smallo{1}\right)n$;
\item\label{ass:2}
$\distributionalLimit{A_{\rootSet}}{G_1}$ and $\distributionalLimit{A_{\overline{\rootSet}}}{G_2}$.
\end{enumerate}
Then we have
\begin{align*}
\distributionalLimit{\unbiased{A}}{a~G_1+\left(1-a\right)G_2}.
\end{align*} 
\end{lem}
\begin{proof}
Assume $\unbiased{A}=\rootedGraph{A}{\root}$, where $r=r(n)\ur[n]$. We denote by $\tilde{\cl}(n)\subseteq\cl(n)$ the subclass containing those graphs $H\in\cl(n)$ satisfying $\left|\rootSet\left(H\right)\right|=\left(a+\smallo{1}\right)n$. By assumption \ref{ass:1} we have \whp\ $A\in\tilde{\cl}(n)$. Using that we obtain for each rooted graph $\detGraph$ and $\radius\in \N$

\begin{align}\label{eq:2}
\prob{\ballP{\ell}{\unbiased{A}}\isomorphic\detGraph}&=~~~~\sum_{H\in\tilde{\cl}(n)}\prob{A=H,\root\in \rootSet(H)}\condprob{\ballP{\ell}{\unbiased{A}}\isomorphic\detGraph}{A=H,\root\in \rootSet(H)}\nonumber
\\&~~~~+\sum_{H\in\tilde{\cl}(n)}\prob{A=H,\root\in \overline{\rootSet}(H)}\condprob{\ballP{\ell}{\unbiased{A}}\isomorphic\detGraph}{A=H,\root\in \overline{\rootSet}(H)}+\smallo{1}.
\end{align}
By definition of $\tilde{\cl}$ we have uniformly over all $H\in\tilde{\cl}(n)$
\begin{align}\label{eq:5}
\prob{A=H,\root\in \rootSet(H)}=\prob{A=H}\condprob{\root\in\rootSet(H)}{A=H}=\left|\cl(n)\right|^{-1}\frac{\left|\rootSet(H)\right|}{n}= \left(a+\smallo{1}\right)\left|\cl(n)\right|^{-1}.
\end{align}
Moreover, we have
\begin{align}\label{eq:6}
\condprob{\ballP{\ell}{\unbiased{A}}\isomorphic\detGraph}{A=H,\root\in \rootSet(H)}=\condprob{\ballP{\ell}{A_{\rootSet}}\isomorphic\detGraph}{A=H}=\left|\cl(n)\right|\prob{\ballP{\ell}{A_{\rootSet}}\isomorphic\detGraph, A=H}.
\end{align}
Combining (\ref{eq:5}) and (\ref{eq:6}) yields
\begin{align}\label{eq:7}
\sum_{H\in\tilde{\cl}(n)}\prob{A=H,\root\in \rootSet(H)}\condprob{\ballP{\ell}{\unbiased{A}}\isomorphic\detGraph}{A=H,\root\in \rootSet(H)}&=\left(a+\smallo{1}\right)\sum_{H\in\tilde{\cl}(n)}\prob{\ballP{\ell}{A_{\rootSet}}\isomorphic\detGraph, A=H}\nonumber
\\
&=\left(a+\smallo{1}\right)\prob{\ballP{\ell}{A_{\rootSet}}\isomorphic\detGraph, A\in\tilde{\cl}}\nonumber
\\
&
=\left(a+\smallo{1}\right)\prob{\ballP{\ell}{A_{\rootSet}}\isomorphic\detGraph},
\end{align}
where we used in the last equality that \whp\ $A\in\tilde{\cl}$. Analogously, we obtain
\begin{align}\label{eq:8}
	\sum_{H\in\tilde{\cl}(n)}\prob{A=H,\root\in \overline{\rootSet}(H)}\condprob{\ballP{\ell}{\unbiased{A}}\isomorphic\detGraph}{A=H,\root\in \overline{\rootSet}(H)}=\left(1-a+\smallo{1}\right)\prob{\ballP{\ell}{A_{\overline{\rootSet}}}\isomorphic\detGraph}.
\end{align}
Next we plug (\ref{eq:7}) and (\ref{eq:8}) in equation (\ref{eq:2}) to get
\begin{align*}
\prob{\ballP{\ell}{\unbiased{A}}\isomorphic\detGraph}&=\left(a+\smallo{1}\right)\prob{\ballP{\ell}{A_{\rootSet}}\isomorphic\detGraph}+\left(1-a+\smallo{1}\right)\prob{\ballP{\ell}{A_{\overline{\rootSet}}}\isomorphic\detGraph}+\smallo{1}
\\
&=a~\prob{\ballP{\ell}{G_1}\isomorphic\detGraph}+\left(1-a\right)\prob{\ballP{\ell}{G_2}\isomorphic\detGraph}+\smallo{1},
\end{align*}
as $\distributionalLimit{A_{\rootSet}}{G_1}$ and $\distributionalLimit{A_{\overline{\rootSet}}}{G_2}$ by assumption \ref{ass:2}. As this is true for all rooted graphs $\detGraph$ and $\radius\in\N$, we obtain $\distributionalLimit{\unbiased{A}}{a~G_1+\left(1-a\right)G_2}$, as desired.
\end{proof}

\subsection{Conditional random rooted graphs}\label{sub:conditional_random_graphs}
Instead of considering $\planargraph=\planargraph(n,m)$ directly, we will later analyse the local structure of $\planargraph$ conditioned that certain properties hold, e.g. $\planargraph$ has a given graph $C$ as its core. We expect that if \lq most\rq\ of these \lq conditional\rq\ random rooted graphs have the same local limit $G$, then $G$ should also be the local limit of $\planargraph$. In the following we make that more precise by considering conditional random rooted graphs, which generalise the concept of conditional random graphs introduced in \cite[Section 3.4]{cycles}.
\begin{definition}
Let $\cl$ be a class of graphs, $\rootSet=\rootSet(H)\subseteq \vertexSet{H}$ a non--empty subset for each graph $H\in\cl$, $\imageSet$ a set, and $\func:\cl\to \imageSet$ a function. We call a sequence $\seq=\left(\term_n\right)_{n\in\N}$ {\em feasible} for $\left(\cl, \func\right)$ if for each $n\in\N$ there exists a graph $H\in\cl(n)$ such that $\func(H)=\term_n$. Furthermore, for each $n\in\N$ we denote by $\left(\condGraph{A_{\rootSet}}{\seq}\right)(n)$ the random rooted graph $\rootedGraph{G}{\root}$ which is constructed in two steps as follows:
\begin{itemize}
	\item
	$G$ is chosen uniformly at random from the set $\setbuilder{H\in\cl(n)}{\func(H)=\term_n}$;
	\item
	Given the realisation $H$ of $G$, the root $\root$ is chosen uniformly at random from $\rootSet(H)$.
\end{itemize} 
We will often omit the dependence on $n$ and write just $\condGraph{A_{\rootSet}}{\seq}$.
\end{definition}

\begin{lem}\label{lem:conditional}
Let $\cl$ be a class of graphs, $\rootSet=\rootSet(H)\subseteq\vertexSet{H}$ a non--empty subset for each graph $H\in\cl$, $\imageSet$ a set, and $\func:\cl\to \imageSet$. Suppose that $G$ is a random rooted graph such that $\distributionalLimit{\condGraph{A_{\rootSet}}{\seq}}{G}$ for each sequence $\seq$ that is feasible for $\left(\cl, \func\right)$. Then we have $\distributionalLimit{A_{\rootSet}}{G}$.
\end{lem}
\begin{proof}
Assume $A_{\rootSet}=\rootedGraph{A}{\root_{\rootSet}}$, where $A=A(n)\ur\cl(n)$. Furthermore, let $\detGraph$ be a rooted graph and $\radius\in\N$. We define $\seq^\ast=\left(\term_n^\ast\right)_{n\in\N}$ such that for each $n\in \N$ the probability $\condprob{\ballP{\radius}{A_{\rootSet}(n)}\isomorphic \detGraph}{\func(A)=\term}$ is maximised for $\term=\term_n^\ast$ (among all $\term\in \imageSet$ for which there is a graph $H\in\cl(n)$ with $\func(H)=\term$). Then we have for $n\in\N$
\begin{align}\label{eq:9}
\prob{\ballP{\radius}{A_{\rootSet}}\isomorphic \detGraph}&=\sum_{\term\in \imageSet}\prob{\func(A)=\term}\condprob{\ballP{\radius}{A_{\rootSet}}\isomorphic \detGraph}{\func(A)=\term}\nonumber
\\
&
\leq\sum_{\term\in \imageSet}\prob{\func(A)=\term}\condprob{\ballP{\radius}{A_{\rootSet}}\isomorphic \detGraph}{\func(A)=\term_n^\ast}\nonumber
\\
&
=\condprob{\ballP{\radius}{A_{\rootSet}}\isomorphic \detGraph}{\func(A)=\term_n^\ast}\nonumber
\\
&
=\prob{\ballP{\radius}{\condGraph{A_{\rootSet}}{\seq^\ast}}\isomorphic \detGraph}.
\end{align}
We note that the sequence $\seq^\ast$ is feasible and therefore we have $\distributionalLimit{\condGraph{A_{\rootSet}}{\seq^\ast}}{G}$. Using that in (\ref{eq:9}) yields $\prob{\ballP{\radius}{A_{\rootSet}}\isomorphic \detGraph}\leq\prob{\ballP{\radius}{G}\isomorphic \detGraph}+\smallo{1}$. Analogously, we obtain $\prob{\ballP{\radius}{A_{\rootSet}}\isomorphic \detGraph}\geq \prob{\ballP{\radius}{G}\isomorphic \detGraph}+\smallo{1}$ and therefore $\distributionalLimit{A_{\rootSet}}{G}$, as desired.
\end{proof}

\subsection{Contiguity of random rooted graphs}
To express that the local structures of two random rooted graphs are asymptotically similar, we use the concept of contiguity.  
\begin{definition}
For each $n\in\N$, let $G_1=G_1(n)$ and $G_2=G_2(n)$ be random rooted graphs. We say that $G_1$ and $G_2$ are {\em contiguous}, denoted by $\contiguous{G_1}{G_2}$, if for all rooted graphs $\detGraph$ and $\radius\in\N$
\begin{align*}
\Big|\prob{\ballP{\radius}{G_1}\isomorphic\detGraph}-\prob{\ballP{\radius}{G_2}\isomorphic\detGraph}\Big|=\smallo{1}.
\end{align*}
\end{definition}

As expected, the local limits of two contiguous random rooted graphs coincide provided they exist.
\begin{lem}\label{lem:contiguous1}
For each $n\in\N$, let $G_1=G_1(n)$ and $G_2=G_2(n)$ be random rooted graphs such that $\contiguous{G_1}{G_2}$. If there exists a random rooted graph $G$ such that $\distributionalLimit{G_1}{G}$, then also $\distributionalLimit{G_2}{G}$.
\end{lem}
\begin{proof}
Let $\detGraph$ be a rooted graph and $\radius\in\N$. We have
\begin{align*}
\prob{\ballP{\radius}{G_2}\isomorphic\detGraph}=\prob{\ballP{\radius}{G_1}\isomorphic\detGraph}+\smallo{1}=\prob{\ballP{\radius}{G}\isomorphic\detGraph}+\smallo{1},
\end{align*}
which shows the statement.
\end{proof}

Finally, we again consider the random $\rootSet$--rooted graph $A_{\rootSet}$ introduced in \Cref{def:R_rooted}. Roughly speaking, the next lemma states that the local structure of $A_{\rootSet}$ is \lq robust\rq\ against minor changes of the class $\cl$ or the root set $\rootSet$.
\begin{lem}\label{lem:contiguous2}
Let $\cl$ be a class of graphs, $A=A(n)\ur\cl(n)$, and for each graph $H\in\cl$ let $\rootSet=\rootSet(H)\subseteq\vertexSet{H}$ be a non--empty subset.
\begin{enumerate}[label=\normalfont(\roman*)]
\item\label{lem:contiguous2A}
Let $\tilde{\cl}\subseteq\cl$ be a subclass and $\tilde{A}=\tilde{A}(n)\ur\tilde{\cl}(n)$. If \whp\ $A\in\tilde{\cl}$, then $\contiguous{A_\rootSet}{\tilde{A}_\rootSet}$.
\item\label{lem:contiguous2B}
For each $H\in\cl$ let $\rootSet'=\rootSet'(H)\subseteq\vertexSet{H}$ be a non--empty subset. If \whp\ $\left|\symmetricDifference{\rootSet(A)}{\rootSet'(A)}\right|=\smallo{\left|\rootSet(A)\right|}$, then $\contiguous{A_{\rootSet}}{A_{\rootSet'}}$.
\end{enumerate}
\end{lem}
\begin{proof}
Let $\detGraph$ be a rooted graph and $\radius\in\N$. Conditioned on the event $A\in\tilde{\cl}$ the random rooted graph $A_\rootSet$ is distributed like $\tilde{A}_{\rootSet}$. Hence, we obtain
\begin{align*}
\prob{\ballP{\radius}{A_{\rootSet}}\isomorphic\detGraph}&=\prob{A\in \tilde{\cl}}\condprob{\ballP{\radius}{A_{\rootSet}}\isomorphic\detGraph}{A\in \tilde{\cl}}+\smallo{1}
\\
&=\left(1-\smallo{1}\right)\prob{\ballP{\radius}{\tilde{A}_{\rootSet}}\isomorphic\detGraph}+\smallo{1}
\\
&=\prob{\ballP{\radius}{\tilde{A}_{\rootSet}}\isomorphic\detGraph}+\smallo{1},
\end{align*}
which shows \ref{lem:contiguous2A}. For \ref{lem:contiguous2B} let $A_\rootSet=\rootedGraph{A}{\root_{\rootSet}}$ and $A_{\rootSet'}=\rootedGraph{A}{\root_{\rootSet'}}$, i.e. $\root_{\rootSet}\ur\rootSet(A)$ and $\root_{\rootSet'}\ur\rootSet'(A)$. Since \whp\ $\left|\symmetricDifference{\rootSet(A)}{\rootSet'(A)}\right|=\smallo{\left|\rootSet(A)\right|}$, we have \whp\ $\left|\rootSet(A)\cap\rootSet'(A)\right|=\left(1+\smallo{1}\right)\left|\rootSet(A)\right|=\left(1+\smallo{1}\right)\left|\rootSet'(A)\right|$. Hence, \whp\ $\root_{\rootSet}, \root_{\rootSet'}\in \rootSet(A)\cap\rootSet'(A)$. Using that we obtain
\begin{align*}
\prob{\ball{\radius}{A}{\root_{\rootSet}}\isomorphic\detGraph}&=\condprob{\ball{\radius}{A}{\root_{\rootSet}}\isomorphic\detGraph}{\root_{\rootSet}\in \rootSet(A)\cap\rootSet'(A)}+\smallo{1}
\\
&=\condprob{\ball{\radius}{A}{\root_{\rootSet'}}\isomorphic\detGraph}{\root_{\rootSet'}\in \rootSet(A)\cap\rootSet'(A)}+\smallo{1}
\\
&=\prob{\ball{\radius}{A}{\root_{\rootSet'}}\isomorphic\detGraph}+\smallo{1},
\end{align*}
which proves \ref{lem:contiguous2B}.
\end{proof}

\section{Proof strategy for main results}\label{sec:strategy}
\subsection{Local structure of the \ER\ random graph}\label{sub:strategy_er}
In order to show \Cref{thm:main1} we will use a result of Britikov \cite{uni} which implies that the random planar graph $\planargraph(n,m)$ \lq behaves\rq\ similarly like the \ER\ random graph $G(n,m)$ as long as $m\leq n/2+\bigo{n^{2/3}}$(see \Cref{thm:non_complex}). Then we will obtain \Cref{thm:main1} by analysing the local structure of $G(n,m)$ (see \Cref{lem:local_er}).

\subsection{Global and local structure of the random planar graph}
If $m$ is as in \Cref{thm:main2,thm:main3,thm:main4} we cannot transfer asymptotic results from $G(n,m)$ to $\planargraph=\planargraph(n,m)$ any more. Roughly speaking, we will show the following instead. The largest component $\Largestcomponent=\largestcomponent{\planargraph}$ of $\planargraph$ consists of a family of trees which is connected via a \lq small\rq\ graph of minimum degree at least two. Thus, a randomly chosen root $\root_L$ from $\Largestcomponent$ will typically lie \lq far away\rq\ from all vertices which are not in the same tree as $\root_L$. This will imply that the local limit of $\Largestcomponent$ coincides with that of a random tree, which is known to be the Skeleton Tree $\skeletonTree$ (see \Cref{thm:limit_tree1}). In addition, we will see that the graph $\Rest=P\setminus\Largestcomponent$ outside the largest component $\Largestcomponent$ \lq behaves\rq\ asymptotically like the \ER\ random graph $G(n_U, n_U/2)$ for a suitable $n_U=n_U(n)\in\N$. Thus, the local limit of $\Rest$ will be $\gwt{1}$, as it is well known that this is the case for $G(n_U, n_U/2)$. Finally for the local structure of $\planargraph$, we use the known fact that \whp\ $\numberVertices{\Largestcomponent}=\left(c-1+\smallo{1}\right)n$, where $c:=\lim_{n\to\infty} 2m/n$ (see \Cref{thm:internal_structure}\ref{thm:internal_structure3}). Thus, the probability that the root $\root$ is in the largest component tends to $(c-1)$ and the local limit of $P=\Largestcomponent \cup\Rest$ will be a \lq linear combination\rq\ of the local limit of $\Largestcomponent$, which is $\skeletonTree$, and that of $\Rest$, which is $\gwt{1}$, i.e. $\left(c-1\right)\skeletonTree+\left(2-c\right)\gwt{1}$. We refer to \Cref{fig:local_structure_planar} for an illustration of the \lq typical\rq\ structure of $\planargraph(n,m)$.

\begin{figure}[t]
	\begin{tikzpicture}[scale=0.75, line width=0.4mm, every node/.style={circle,fill=gray, inner sep=0, minimum size=0.22cm}]
		\node (A) at (0,0) [draw,rectangle]  {};
		\node (B) at (0.5,-0.95) [draw,rectangle]  {};
		\node (C) at (1.5,-1.4) [draw,rectangle]  {};
		\node (D) at (2.5,-0.95) [draw,rectangle]  {};
		\node (E) at (3,0) [draw,rectangle]  {};
		\node (F) at (2.5,0.95) [draw,rectangle]  {};
		\node (G) at (1.5,1.4) [draw,rectangle]  {};
		\node (H) at (0.5,0.95) [draw,rectangle]  {};
		\draw[-] (A) to (B);
		\draw[-] (B) to (C);
		\draw[-] (C) to (D);
		\draw[-] (D) to (E);
		\draw[-] (E) to (F);
		\draw[-] (F) to (G);
		\draw[-] (G) to (H);
		\draw[-] (H) to (A);
		\draw[-] (C) to (F);
		\draw[-] (C) to (H);
		
		\node (I) at (4.2,0) [draw,rectangle]  {};
		\node (J) at (5,1) [draw,rectangle]  {};
		\node (K) at (6.2,0.7) [draw,rectangle]  {};
		\node (L) at (6.2,-0.7) [draw,rectangle]  {};
		\node (M) at (5,-1) [draw,rectangle]  {};
		\draw[-] (I) to (J);
		\draw[-] (J) to (K);
		\draw[-] (K) to (L);
		\draw[-] (L) to (M);
		\draw[-] (M) to (I);
		\draw[-] (E) to (I);
		
		\node (N) at (5.8,-1.7) [draw]  {};
		\node (O) at (5.95,-1.15) [draw]  {};
		\node (P) at (6.6,-1.3) [draw]  {};
		\node (Q) at (7.4,-1.1) [draw]  {};
		\node (R) at (8.1,-1.5) [draw]  {};
		\node (S) at (8.1,-0.7) [draw]  {};
		\node (T) at (8.8,-0.4) [draw]  {};
		\node (U) at (8.8,-1.0) [draw]  {};
		\node (V) at (9.5,-0.7) [draw]  {};
		\node (W) at (9.5,-0.1) [draw]  {};
		\node (X) at (9.5,-1.3) [draw]  {};
		\node (Y) at (10.2,-0.7) [draw]  {};
		\node (Z) at (10.2,-1.3) [draw]  {};
		\draw[-] (M) to (N);
		\draw[-] (M) to (O);
		\draw[-] (N) to (P);
		\draw[-] (P) to (Q);
		\draw[-] (Q) to (R);
		\draw[-] (Q) to (S);
		\draw[-] (S) to (T);
		\draw[-] (S) to (U);
		\draw[-] (U) to (V);
		\draw[-] (U) to (X);
		\draw[-] (T) to (W);
		\draw[-] (V) to (Y);
		\draw[-] (X) to (Z);

		\node (1) at (3.3,-0.6) [draw]  {};
		\node (2) at (3.3,-1.3) [draw]  {};
		\node (3) at (4,-0.8) [draw]  {};
		\node (4) at (4,-1.3) [draw]  {};
		\node (5) at (4,-1.8) [draw]  {};
		\node (6) at (4.7,-1.8) [draw]  {};
		\draw[-] (D) to (1);
		\draw[-] (D) to (2);
		\draw[-] (2) to (3);
		\draw[-] (2) to (4);
		\draw[-] (2) to (5);
		\draw[-] (5) to (6);
		
		\node (7) at (7,1.1) [draw]  {};
		\node (8) at (7,0.35) [draw]  {};
		\node (9) at (7.7,1.1) [draw]  {};
		\draw[-] (K) to (7);
		\draw[-] (K) to (8);
		\draw[-] (7) to (9);
		
		\node (10) at (0.6,-1.5) [draw]  {};
		\node (11) at (-0.2,-1.2) [draw]  {};
		\node (12) at (-0.9,-0.8) [draw]  {};
		\node (13) at (-0.9,-1.5) [draw]  {};
		\node (14) at (-1.6,-1.7) [draw]  {};
		\node (15) at (-1.6,-1.2) [draw]  {};
		\node (16) at (-1.6,-0.4) [draw]  {};
		\node (17) at (-1.6,-0.8) [draw]  {};
		\node (18) at (-2.1,-0.2) [draw]  {};
		\node (19) at (-2.1,-0.8) [draw]  {};
		
		\draw[-] (C) to (10);
		\draw[-] (10) to (11);
		\draw[-] (11) to (12);
		\draw[-] (11) to (13);
		\draw[-] (13) to (14);
		\draw[-] (12) to (15);
		\draw[-] (12) to (16);
		\draw[-] (12) to (17);
		\draw[-] (16) to (18);
		\draw[-] (17) to (19);
		
		\node (20) at (4.4,1.2) [draw]  {};
		\node (21) at (5.6,1.2) [draw]  {};
		\draw[-] (J) to (20);
		\draw[-] (J) to (21);
		
		\node (22) at (-0.7,0.3) [draw]  {};
		\node (23) at (-1.3,0.5) [draw]  {};
		\draw[-] (A) to (22);
		\draw[-] (22) to (23);
		
		\node (24) at (3.2,1.2) [draw]  {};
		\draw[-] (F) to (24);
		\node (25) at (12.5,1.4) [draw]  {};
		\node (26) at (12.5,0.8) [draw]  {};
		\node (27) at (12.5,0.2) [draw]  {};
		\node (28) at (12.8,-0.4) [draw]  {};
		\node (29) at (12.1,-0.4) [draw]  {};
		\node (30) at (12.1,-1) [draw]  {};
		\node (31) at (11.8,-1.6) [draw]  {};
		\node (32) at (12.4,-1.6) [draw]  {};
		
		\draw[-] (25) to (26);
		\draw[-] (26) to (27);
		\draw[-] (27) to (28);
		\draw[-] (27) to (29);
		\draw[-] (29) to (30);
		\draw[-] (30) to (31);
		\draw[-] (30) to (32);
		
		\node (33) at (14,1.4) [draw]  {};
		\node (34) at (13.6,0.8) [draw]  {};
		\node (35) at (14.4,0.8) [draw]  {};
		\node (36) at (14.1,0.2) [draw]  {};
		\node (37) at (14.7,0.2) [draw]  {};
		\node (38) at (14.1,-0.4) [draw]  {};
		\draw[-] (33) to (34);
		\draw[-] (33) to (35);
		\draw[-] (35) to (36);
		\draw[-] (35) to (37);
		\draw[-] (36) to (38);
		
		\node (39) at (13.5,-1.2) [draw]  {};
		\node (40) at (14.1,-1.2) [draw]  {};
		\node (41) at (14.7,-0.8) [draw]  {};
		\node (42) at (14.7,-1.6) [draw]  {};
		\draw[-] (39) to (40);
		\draw[-] (40) to (41);
		\draw[-] (40) to (42);
		
		\node (43) at (15.4,1.3) [draw]  {};
		\node (44) at (16.1,1.3) [draw]  {};
		\draw[-] (43) to (44);
		
		\node (45) at (15.4,0.6) [draw]  {};
		\node (46) at (16.1,0.6) [draw]  {};
		\draw[-] (45) to (46);
		
		\node (47) at (15.4,-0.1) [draw]  {};
		\node (48) at (16.1,-0.1) [draw]  {};
		\draw[-] (47) to (48);
		
		\node (49) at (15.4,-0.8) [draw]  {};
		\node (50) at (16.1,-0.8) [draw]  {};
		\draw[-] (49) to (50);
		
		\node (51) at (15.4,-1.5) [draw]  {};
		\node (52) at (16.1,-1.5) [draw]  {};
		\draw[-] (51) to (52);
		
		\node (53) at (17.1,1.4) [draw]  {};
		\node (54) at (17.9,1.4) [draw]  {};
		\node (55) at (17.1,0.8) [draw]  {};
		\node (56) at (17.9,0.8) [draw]  {};
		\node (57) at (17.1,0.2) [draw]  {};
		\node (58) at (17.9,0.2) [draw]  {};
		\node (59) at (17.1,-0.4) [draw]  {};
		\node (60) at (17.9,-0.4) [draw]  {};
		\node (61) at (17.1,-1) [draw]  {};
		\node (62) at (17.9,-1) [draw]  {};
		\node (63) at (17.1,-1.6) [draw]  {};
		\node (64) at (17.9,-1.6) [draw]  {};
	\end{tikzpicture}
	\caption{\lq Typical\rq\ structure of the random planar graph $P(n,m)$ when $m$ is as in \Cref{thm:main2,thm:main3,thm:main4}: The largest component consists of trees which are connected via a \lq small\rq\ graph of minimum degree at least two (with vertices marked by squares). The remaining part behaves similarly like the \ER\ random graph $G(n_U,n_U/2)$ for some $n_U=n_U(n)\in\N$.}
	\label{fig:local_structure_planar}
\end{figure}
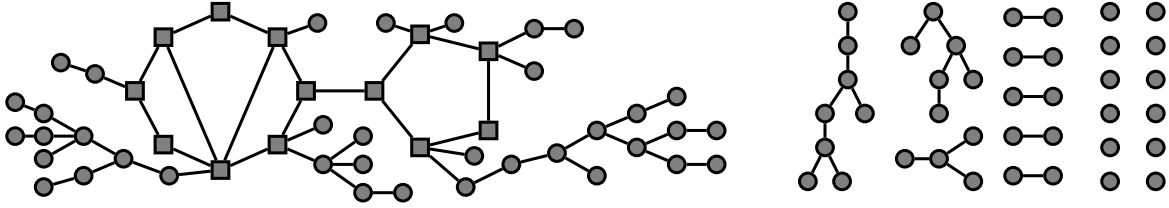

To make that more precise, we use the graph decomposition introduced in \Cref{sub:decomposition}. We recall that the complex part $Q=\complexpart{\planargraph}$ is the union of all components of $P$ that contain at least two cycles and the remaining part $\Restcomplex=\restcomplex{\planargraph}$ is called the non--complex part. Instead of dividing $\planargraph$ into $\Largestcomponent$ and $\Rest$, we actually consider $Q$ and $\Restcomplex$. Then we will use that \whp\ the largest component $\largestcomponent{Q}$ of $Q$ coincides with $\Largestcomponent$ (see \Cref{thm:internal_structure}\ref{thm:internal_structure6}) and $\rest{Q}=Q\setminus\largestcomponent{Q}$ is \lq small\rq\ compared to $Q$ and $\Restcomplex$ (see \Cref{thm:internal_structure}\ref{thm:internal_structure7} and \ref{thm:internal_structure8}). This implies $\contiguous{\unbiased{\Largestcomponent}}{\unbiased{Q}}$ and $\contiguous{\unbiased{\Rest}}{\unbiased{\Restcomplex}}$. In other words, it suffices to determine the local limits of $Q$ and $\Restcomplex$. 

\subsection{Local structure of the complex part}\label{sub:strategy_complex}
To determine the local structure of the complex part $Q$, we will use the concept of conditional random rooted graphs (see \Cref{sub:conditional_random_graphs}): For a given core $C$ and $q\in\N$, we consider $Q$ conditioned on the event that the core $\core{\planargraph}$ of $\planargraph$ is $C$ and that $\numberVertices{Q}=q$. Then $Q$ is distributed like $Q(C,q)$ as defined in \Cref{def:random_complex}, which is a graph chosen uniformly at random from the class of all complex graphs with core $C$ and vertex set $[q]$. As \whp\ $\planargraph$ satisfies $\numberVertices{\core{\planargraph}}=\smallo{\numberVertices{\complexpart{\planargraph}}}$ (see \Cref{thm:internal_structure}\ref{thm:internal_structure2}), it suffices to determine the local structure of $Q(C,q)$ for all $C$ and $q$ fulfilling $\numberVertices{C}=\smallo{q}$. We note that $Q(C,q)$ can be constructed by choosing a random forest $F=F(q, \numberVertices{C})$ on vertex set $[q]$ with $\numberVertices{C}$ tree components such that all vertices from $C$, i.e. vertices in $[\numberVertices{C}]$, lie in different tree components (cf. \Cref{def:random_forest}). Then we obtain $Q(C,q)$ by replacing each vertex $v$ in $C$ by the tree component of $F$ which is rooted at $v$. Let $\root_Q$ be a vertex chosen uniformly at random from $\vertexSet{Q(C,q)}$, $T$ the tree component of $F$ containing $\root_Q$, and $\root_T$ the root of $T$, i.e. $\root_T$ is the unique vertex of $T$ lying in the core $C$. As $\numberVertices{C}=\smallo{q}$, the tree $T$ should be \lq large\rq, i.e. \whp\ $\numberVertices{T}=\smallomega{1}$. We will use this fact to show that then also \whp\ $\distance{T}{\root_Q}{\root_T}=\smallomega{1}$ (see \Cref{lem:limit_forest}\ref{lem:limit_forest2}). In other words, \whp\ there is no vertex from the core lying in the neighbourhood of radius $\radius$ around $\root_Q$. In particular, this shows that $\contiguous{\unbiased{Q(C,q)}}{\unbiased{T}}$. Then we will use that $T$ \lq behaves\rq\ similarly like a random tree and therefore has the same local limit, which is known to be $\skeletonTree$ (see \Cref{thm:limit_tree1} and \Cref{lem:limit_forest}\ref{lem:limit_forest1}). Thus, we obtain $\distributionalLimit{\unbiased{Q(C,q)}}{\skeletonTree}$ (see \Cref{lem:local_random_complex}) and therefore also $\distributionalLimit{\unbiased{\complexpart{\planargraph}}}{\skeletonTree}$ (see \Cref{lem:local_complex}) and $\distributionalLimit{\unbiased{\Largestcomponent}}{\skeletonTree}$.

\subsection{Local structure of the non--complex part}\label{sub:strategy_non_complex}
Similarly, we determine the local structure of the non--complex part $\Restcomplex$. Conditioned on $\numberVertices{\Restcomplex}=n_U$ and $\numberEdges{\Restcomplex}=m_U$ for $n_U, m_U \in\N$, the non--complex part $\Restcomplex$ is distributed like a graph $U(n_U,m_U)$ chosen uniformly at random from the class of all graphs without complex components and $n_U$ vertices and $m_U$ edges as defined in \Cref{def:random_non_complex}. It is known that \whp\ $\numberEdges{U}=\numberVertices{U}/2+\bigo{h\numberVertices{U}^{2/3}}$ for each function $h=h(n)=\smallomega{1}$ (see \Cref{thm:internal_structure}\ref{thm:internal_structure5}). Thus, we can restrict our considerations to the case $m_U=n_U/2+\bigo{n_U^{2/3}}$. In this regime the probability that the \ER\ random graph $G=G(n_U,m_U)$ has no complex component is bounded away from zero (see \Cref{thm:non_complex}). Hence, each graph property that holds \whp\ in $G(n_U,m_U)$ is also true \whp\ in $U(n_U,m_U)$. We observe that this is not enough to deduce $\contiguous{\unbiased{G(n_U,m_U)}}{\unbiased{U(n_U,m_U)}}$. However, we will use the first and second moment method to count the number of vertices $v$ in $G=G(n_U,m_U)$ satisfying $\ball{\radius}{G}{v}\isomorphic\detGraph$ for a fixed rooted graph $\detGraph$ and $\radius\in\N$. More precisely, we will show that \whp\ this number is $\left(1+\smallo{1}\right)\prob{\ballP{\radius}{\gwt{1}}\isomorphic\detGraph}\cdot n_U$ (see \Cref{lem:local_er}). In particular, this is also true for $U(n_U,m_U)$, which immediately implies $\distributionalLimit{\unbiased{U(n_U,m_U)}}{\gwt{1}}$ (see \Cref{coro:local_no_complex}). Hence, we obtain $\distributionalLimit{\unbiased{U}}{\gwt{1}}$ (see \Cref{lem:local_non_complex}) and $\distributionalLimit{\unbiased{\Rest}}{\gwt{1}}$.

\section{Local structure of the \ER\ random graph and the non--complex part}\label{sec:local_er_non_complex}
We start this section with a stronger statement than \Cref{thm:local_er} on the local structure of the \ER\ random graph $G(n,m)$, which will have two main applications. Firstly, we will use \Cref{lem:local_er} in the proof of \Cref{thm:main1}, where the random planar graph $\planargraph=\planargraph(n,m)$ \lq behaves\rq\ similarly like $G(n,m)$ (see \Cref{sub:strategy_er}). Secondly, later in this section \Cref{lem:local_er} will be the starting point of determining the local limit of the non--complex part $\restcomplex{\planargraph}$ of $\planargraph$ as described in \Cref{sub:strategy_non_complex}.
\begin{lem}\label{lem:local_er}
Let $m=m(n)=\alpha n/2$ where $\alpha=\alpha(n)$ tends to a constant $c\geq 0$, $\detGraph$ be a rooted graph, $\radius\in\N$, and $G=G(n,m)\ur\erClass(n,m)$. Then \whp
\begin{align*}
n^{-1}\cdot\Big|\setbuilderBig{v\in\vertexSet{G}}{\ball{\radius}{G}{v}\isomorphic\detGraph}\Big|=\left(1+\smallo{1}\right)\prob{\ballP{\radius}{\gwt{c}}\isomorphic\detGraph}.
\end{align*}
\end{lem}

A version of \Cref{lem:local_er} for the binomial random graph $G(n,p)$ is well known and its proof can be found e.g. in {\cite[Chapter 2]{vdH2}}, but for the sake of completeness we provide a proof of \Cref{lem:local_er} in \Cref{app:local_er}. In the next step, we transfer the statement of \Cref{lem:local_er} to the random graph $U(n,m)$, which is a graph chosen uniformly at random from the class $\nocomplexClass(n,m)$ of all graphs without complex components on vertex set $[n]$ with $m$ edges (cf. \Cref{def:random_non_complex}). To that end, we use the following result of Britikov \cite{uni}.

\begin{thm}[\cite{uni}]\label{thm:non_complex}
	Let $m=m(n)\leq n/2+\bigo{n^{2/3}}$ and $G=G(n,m)\ur \erClass(n,m)$ be the \ER\ random graph. Then 
	\begin{align*}
		\liminf_{n \to \infty}~ \prob{G \text{ has no complex component }}>0.
	\end{align*}
\end{thm}

\Cref{thm:non_complex} implies that each graph property that holds \whp\ in $G(n,m)$ is also true \whp\ in $U(n,m)$ as long as $m\leq n/2+\bigo{n^{2/3}}$. In particular, we can combine it with \Cref{lem:local_er} to deduce the local limit of $U(n,m)$. 

\begin{coro}\label{coro:local_no_complex}
	Let $m=m(n)\leq n/2+\bigo{n^{2/3}}$ be such that $2m/n\to c\in[0,1]$ and $U=U(n,m)\ur\nocomplexClass(n,m)$. Then $\distributionalLimit{\unbiased{U}}{\gwt{c}}$.
\end{coro}
\begin{proof}
Let $\detGraph$ be a rooted graph and $\radius\in\N$. To simplify notation we set $\beta:=\prob{\ballP{\radius}{\gwt{c}}\isomorphic\detGraph}$. We consider the random variable
\begin{align*}
X:=\left|\setbuilder{v\in\vertexSet{U}}{\ball{\radius}{U}{v}\isomorphic\detGraph}\right|.
\end{align*}
By \Cref{lem:local_er} and \Cref{thm:non_complex} we have \whp\ $X=\left(1+\smallo{1}\right)\beta n$ and in particular $\expec{X}=\left(1+\smallo{1}\right)\beta n$. Hence, we obtain
\begin{align*}
\prob{\ballP{\radius}{\unbiased{U}}\isomorphic\detGraph}=n^{-1}\expec{X}=\left(1+\smallo{1}\right)\beta,
\end{align*}
which shows the statement.
\end{proof}

Finally, we use \Cref{coro:local_no_complex} to determine the local limit of the non--complex part $\restcomplex{\planargraph}$ of $\planargraph=\planargraph(n,m)$ under reasonable assumptions. Later we will see that in all considered cases $\planargraph$ satisfies these conditions.

\begin{lem}\label{lem:local_non_complex}
Let $\planargraph=\planargraph(n,m)\ur\planarclass(n,m)$ be the random planar graph and $\Restcomplex=\Restcomplex(\planargraph)$ the non--complex part of $\planargraph$ with  $\numberVerticesOutside=\numberVerticesOutside(\planargraph)$ vertices and $\numberEdgesOutside=\numberEdgesOutside(\planargraph)$ edges. Assume that $m=m(n)$ is such that \whp\ $n_U=\smallomega{1}$, $m_U\leq n_U/2+\bigo{hn_U^{2/3}}$ for each $h=h(n)=\smallomega{1}$, and $2m_U/n_U\to c\in [0,1]$. Then $\distributionalLimit{\planargraph_\Restcomplex}{\gwt{c}}$.
\end{lem}
\begin{proof}
We will use \Cref{lem:conditional}. To that end, let $\cl(n)$ be the subclass of $\planarclass(n,m)$ that contains those graphs $H$ satisfying
\begin{align*}
\numberVerticesOutside(H)&=\smallomega{1},\\
\numberEdgesOutside(H)&\leq \numberVerticesOutside(H)/2+\bigo{\numberVerticesOutside(H)^{2/3}},\\
2\numberEdgesOutside(H)/\numberVerticesOutside(H)&=c+\smallo{1}.
\end{align*}
By assumption, we can choose the implicit constants in the equations above such that $\planargraph\in\cl:=\cup_n\cl(n)$ with a probability of at least $1-\delta$, so as to obtain for each rooted graph $\detGraph$ and $\radius\in\N$
\begin{align}\label{eq:10}
\prob{\ballP{\radius}{\planargraph_U}\isomorphic\detGraph}&=\prob{P\in\cl}\condprob{\ballP{\radius}{\planargraph_U}\isomorphic\detGraph}{P\in\cl}+\prob{P\notin\cl}\condprob{\ballP{\radius}{\planargraph_U}\isomorphic\detGraph}{P\notin\cl}\nonumber
\\
&\leq
\condprob{\ballP{\radius}{\planargraph_U}\isomorphic\detGraph}{P\in\cl}+\delta,
\end{align}
for an arbitrary $\delta>0$. Let $A_U=A_U(n)$ be the random rooted graph such that $A=A(n)\ur\cl(n)$ and given the realisation $H$ of $A$, the root is chosen uniformly at random from $\vertexSet{\restcomplex{H}}$. We define the function $\func$ such that $\func(H)=\left(\numberVerticesOutside(H), \numberEdgesOutside(H)\right)$ for each $H\in\cl$ and let $\seq=\left(\nu_n, \mu_n\right)_{n\in\N}$ be a sequence that is feasible for $(\cl, \func)$. We note that the local structure of $\condGraph{A_U}{\seq}$ is distributed like that of $\unbiased{U\left(\nu_n, \mu_n\right)}$, i.e. for each rooted graph $\detGraph$ and $\radius\in\N$ we have 
\begin{align*}
\prob{\ballP{\radius}{\condGraph{A_U}{\seq}}\isomorphic \detGraph}=\prob{\ballP{\radius}{\unbiased{U\left(\nu_n, \mu_n\right)}}\isomorphic \detGraph}.
\end{align*}
Thus, we obtain by \Cref{coro:local_no_complex} that $\distributionalLimit{\condGraph{A_U}{\seq}}{\gwt{c}}$. Combining that with \Cref{lem:conditional} yields $\distributionalLimit{A_U}{\gwt{c}}$, i.e. $\prob{\ballP{\radius}{A_U}\isomorphic\detGraph}=\prob{\ballP{\radius}{\gwt{c}}\isomorphic\detGraph}+\smallo{1}$. Conditioned on $\planargraph\in\cl$, $\planargraph$ is distributed like $A$. Hence, we have $\condprob{\ballP{\radius}{\planargraph_U}\isomorphic\detGraph}{P\in\cl}=\prob{\ballP{\radius}{A_U}\isomorphic\detGraph}=\prob{\ballP{\radius}{\gwt{c}}\isomorphic\detGraph}+\smallo{1}$. Plugging this in (\ref{eq:10}) yields
\begin{align*}
\prob{\ballP{\radius}{\planargraph_U}\isomorphic\detGraph}
\leq
\prob{\ballP{\radius}{\gwt{c}}\isomorphic\detGraph}+\smallo{1}+\delta.
\end{align*}
As $\delta>0$ is arbitrary, we obtain $\prob{\ballP{\radius}{\planargraph_U}\isomorphic\detGraph}\leq \prob{\ballP{\radius}{\gwt{c}}\isomorphic\detGraph}+\smallo{1}$. Analogously, we can show $\prob{\ballP{\radius}{\planargraph_U}\isomorphic\detGraph}\geq \prob{\ballP{\radius}{\gwt{c}}\isomorphic\detGraph}+\smallo{1}$. Thus, we have $\distributionalLimit{\planargraph_U}{\gwt{c}}$, as desired.
\end{proof}

\section{Local structure of the complex part}\label{sec:local_complex}
In this section we determine the local structure of the complex part $\complexpart{\planargraph}$ of $\planargraph=\planargraph(n,m)$ as described in \Cref{sub:strategy_complex}. We recall that given a graph $H$, the complex part $\complexpart{H}$ is the union of all components of $H$ with at least two cycles. In \Cref{sub:strategy_complex} we saw that there is a close relation between the complex part and random forests. Therefore, we start with the known result that two randomly chosen vertices in a random tree are \lq far away\rq\ from each other.

\begin{thm}[\cite{distance_random_trees}]\label{thm:limit_tree2}
Let $T=T(n)$ be a random tree and $\root=\root(n), \root'=\root'(n)\ur\vertexSet{T}$ be chosen independently and uniformly at random from $\vertexSet{T}$. Then \whp\ $~\distance{T}{\root}{\root'}=\smallomega{n^{1/3}}$.
\end{thm}
In \cite{distance_random_trees} a more general version than \Cref{thm:limit_tree2} is actually shown, namely that $\distance{T}{\root}{\root'}$ is concentrated around $n^{1/2}$. Next, we consider the random forest $\forest(n,\ntrees)$ defined as follows.
\begin{definition}\label{def:random_forest}
	Given $n, \ntrees\in\N$, let $\forestClass(n,\ntrees)$ be the class of all forests on vertex set $[n]$ having $\ntrees$ trees as components such that the vertices $1, \ldots, \ntrees$ lie all in different tree components. We denote by $\forest(n,\ntrees)\ur\forestClass(n,\ntrees)$ a forest chosen uniformly at random from $\forestClass(n,\ntrees)$ and call the vertices $1, \ldots, \ntrees$ the {\em roots} of the tree components of $\forest(n,\ntrees)$.
\end{definition}
In our applications the number of tree components will always be \lq small\rq, in other words $t=\smallo{n}$. Thus, a randomly chosen vertex $r\ur\vertexSet{\forest}$ should lie in a \lq large\rq\ tree component, i.e. $\numberVertices{T'}=\smallomega{1}$, where $T'$ is the tree component of $\forest$ containing $r$. Therefore, we expect that the local structures of $T'$ and $\forest$ should be close to that of a random tree. In the following Lemma we will show that this is indeed true. In view of \Cref{thm:limit_tree2} we also expect that the distance between $r$ and the root of $T'$ should be $\smallomega{\numberVertices{T'}^{1/3}}$, so in particular tending to infinity.
\begin{lem}\label{lem:limit_forest}
Let $t=t(n)=\smallo{n}$ and $\forest=\forest(n,t)\ur\forestClass(n,t)$. Moreover, let $\root=\root(n)\ur\vertexSet{\forest}$ and $\root'=\root'(n)\in[t]$ be the root of the tree component in $\forest$ that contains $\root$. Then we have
\begin{enumerate}[label=\normalfont(\roman*)]
\item\label{lem:limit_forest1}
$\distributionalLimit{\unbiased{\forest}}{\skeletonTree}$;
\item\label{lem:limit_forest2}
\whp\ $~\distance{\forest}{\root}{\root'}=\smallomega{1}$.
\end{enumerate}
\end{lem}
\begin{proof}
Without loss of generality, we can assume $\unbiased{\forest}=\rootedGraph{\forest}{\root}$. Let $T'$ be the tree component of $F$ that contains $r$, $h=h(n):=n/t=\smallomega{1}$, and $T(k)$ be the random tree on $k\in \N$ vertices. Conditioned on the event $\numberVertices{T'}=k$, the random rooted tree $\rootedGraph{T'}{\root}$ is distributed like $\unbiased{T(k)}$. Thus, we obtain for each rooted graph $\detGraph$ and $\radius\in \N$
\begin{align}\label{eq:11}
\prob{\ball{\radius}{T'}{\root}\isomorphic\detGraph}&=
\sum_{k=1}^{n}\prob{\numberVertices{T'}=k}\condprob{\ball{\radius}{T'}{\root}\isomorphic\detGraph}{\numberVertices{T'}=k}\nonumber
\\
&=
\sum_{k=1}^{n}\prob{\numberVertices{T'}=k}\prob{\ballP{\radius}{\unbiased{T(k)}}\isomorphic\detGraph}\nonumber
\\
&=
\sum_{k\leq \sqrt{h}}\prob{\numberVertices{T'}=k}\prob{\ballP{\radius}{\unbiased{T(k)}}\isomorphic\detGraph}+
\sum_{k>\sqrt{h}}\prob{\numberVertices{T'}=k}\prob{\ballP{\radius}{\unbiased{T(k)}}\isomorphic\detGraph}.
\end{align}
We note that $F$ has at least $n-t\sqrt{h}=\left(1-\smallo{1}\right)n$ vertices that lie in a tree component with more than $\sqrt{h}$ vertices. Hence, \whp\ $\numberVertices{T'}>\sqrt{h}$, which implies 
\begin{align}\label{eq:12}
\sum_{k\leq \sqrt{h}}\prob{\numberVertices{T'}=k}\prob{\ballP{\radius}{\unbiased{T(k)}}\isomorphic\detGraph}=\smallo{1}.
\end{align}
Due to \Cref{thm:limit_tree1} we have $\prob{\ballP{\radius}{\unbiased{T(k)}}\isomorphic\detGraph}=\prob{\ballP{\radius}{\skeletonTree}\isomorphic\detGraph}+\smallo{1}$ uniformly over all $k>\sqrt{h}$. Hence, we obtain
\begin{align}\label{eq:13}
\sum_{k>\sqrt{h}}\prob{\numberVertices{T'}=k}\prob{\ballP{\radius}{\unbiased{T(k)}}\isomorphic\detGraph}
&=
\big(\prob{\ballP{\radius}{\skeletonTree}\isomorphic\detGraph}+\smallo{1}\big)\sum_{k>\sqrt{h}}\prob{\numberVertices{T'}=k}\nonumber
\\
&=\big(\prob{\ballP{\radius}{\skeletonTree}\isomorphic\detGraph}+\smallo{1}\big)\prob{\numberVertices{T'}>\sqrt{h}}\nonumber
\\
&=
\prob{\ballP{\radius}{\skeletonTree}\isomorphic\detGraph}+\smallo{1}.
\end{align}
Using (\ref{eq:12}) and (\ref{eq:13}) in equation (\ref{eq:11}) yields $\prob{\ball{\radius}{T'}{\root}\isomorphic\detGraph}=\prob{\ballP{\radius}{\skeletonTree}\isomorphic\detGraph}+\smallo{1}$. This shows statement \ref{lem:limit_forest1}, as the local structure of $\unbiased{\forest}=\rootedGraph{\forest}{\root}$ is that of $\rootedGraph{T'}{\root}$.

Next, we prove statement \ref{lem:limit_forest2} in a similar way. Conditioned on the event $\left\{\numberVertices{T'}=k, \root\notin [\ntrees]\right\}$, the distance $\distance{F}{\root}{\root'}$ between $\root$ and the root $\root'$ of $T'$ is distributed like the distance between two randomly chosen distinct vertices in $T(k)$, which is \whp\ larger than $k^{1/3}$ due to \Cref{thm:limit_tree2}. Hence, we obtain uniformly over all $k>\sqrt{h}$
\begin{align*}
\condprob{\distance{F}{\root}{\root'}\geq h^{1/6}}{\numberVertices{T'}=k, \root\notin [\ntrees]}\geq \condprob{\distance{F}{\root}{\root'}\geq k^{1/3}}{\numberVertices{T'}=k, \root\notin [\ntrees]}=1-\smallo{1}.
\end{align*}
Using that yields
\begin{align*}
\prob{\distance{F}{\root}{\root'}\geq h^{1/6}}
&\geq \prob{\distance{F}{\root}{\root'}\geq h^{1/6}, \root\notin [\ntrees]}
\\
&=\sum_{k=1}^{n}\prob{\numberVertices{T'}=k, \root\notin [\ntrees]}\condprob{\distance{F}{\root}{\root'}\geq h^{1/6}}{\numberVertices{T'}=k, \root\notin [\ntrees]}
\\
&\geq
\sum_{k>\sqrt{h}}\prob{\numberVertices{T'}=k, \root\notin [\ntrees]}\condprob{\distance{F}{\root}{\root'}\geq h^{1/6}}{\numberVertices{T'}=k, \root\notin [\ntrees]}
\\
&=
\left(1-\smallo{1}\right)\prob{\numberVertices{T'}>\sqrt{h}, \root\notin [\ntrees]}
\\
&=1-\smallo{1},
\end{align*}
where we used in the last equality that \whp\ $\numberVertices{T'}>\sqrt{h}$ and $\prob{\root\notin [\ntrees]}=\left(n-t\right)/n=1-\smallo{1}$.
This proves statement \ref{lem:limit_forest2}.
\end{proof}

We recall that given a core $C$ and $q\in\N$, we denote by $Q=Q(C,q)$ a graph chosen uniformly at random from the class of all complex graphs with core $C$ and vertex set $[q]$. Furthermore, we can construct $Q$ by choosing a forest $\forest=\forest(q, \numberVertices{C})$ and replacing each vertex $v$ in $C$ by the tree component of $\forest$ which is rooted at $v$. Assuming $\numberVertices{C}=\smallo{q}$ we can use \Cref{lem:limit_forest}. Hence, a randomly chosen vertex $r\ur\vertexSet{Q}$ will typically lie \lq far away\rq\ from all vertices in $C$. Using that we will show that the local limit of $Q$ coincides with that of $\forest$.
\begin{lem}\label{lem:local_random_complex}
For each $n\in\N$, let $C=C(n)$ be a core, $q=q(n)\in\N$, and $Q=Q(C,q)$ be the random complex part with core $C$ and vertex set $[q]$. If $\numberVertices{C}=\smallo{q}$, then $\distributionalLimit{\unbiased{Q}}{\skeletonTree}$.
\end{lem}
\begin{proof}
We assume that $\unbiased{Q}=\rootedGraph{Q}{r}$, i.e. $r\ur\vertexSet{Q}$, and that $Q$ can be constructed by choosing a random forest $\forest=\forest(q, \numberVertices{C})$ and replacing each vertex $v$ in $C$ by the tree component of $\forest$ rooted at $v$. Due to \Cref{lem:limit_forest}\ref{lem:limit_forest2} \whp\ the distance from $r$ to the root of the tree component in $F$ containing $r$ is larger than a fixed constant $\radius\in\N$. Thus, we get \whp\ $\ball{\radius}{Q}{\root}=\ball{\radius}{\forest}{\root}$, which implies $\contiguous{\unbiased{Q}}{\unbiased{\forest}}$. Combining that with \Cref{lem:contiguous1} and \Cref{lem:limit_forest}\ref{lem:limit_forest1} yields the statement.
\end{proof}

In all our applications the core $\core{\planargraph}$ will be \lq small\rq\ compared to the complex part $\complexpart{\planargraph}$, i.e. \whp\ $\numberVertices{\core{\planargraph}}=\smallo{\numberVertices{\complexpart{\planargraph}}}$. In such a case we can apply \Cref{lem:local_random_complex} to deduce the local structure of $\complexpart{\planargraph}$.
\begin{lem}\label{lem:local_complex}
Let $\planargraph=\planargraph(n,m)\ur\planarclass(n,m)$ be the random planar graph and $Q=\complexpart{\planargraph}$ the complex part of $\planargraph$. Assume that $m=m(n)$ is such that \whp\ $\numberVertices{\core{\planargraph}}=\smallo{\numberVertices{\complexpart{\planargraph}}}$. Then $\distributionalLimit{\planargraph_Q}{\skeletonTree}$.
\end{lem}
\begin{proof}
We will use \Cref{lem:conditional}. Let $\cl(n)$ be the subclass of $\planarclass(n,m)$ containing those graphs $H$ satisfying $\numberVertices{\core{H}}=\smallo{\numberVertices{\complexpart{H}}}$. By assumption we have \whp\ $\planargraph\in \cl:=\cup_n\cl(n)$. Let $A_Q=A_Q(n)$ be the random rooted graph such that $A=A(n)\ur\cl(n)$ and given the realisation $H$ of $A$, the root is chosen uniformly at random from $\vertexSet{\complexpart{H}}$. We define the function $\func$ such that $\func(H)=\left(\core{H}, \numberVertices{\complexpart{H}}\right)$ for each $H\in\cl$ and let $\seq=\left(C_n, q_n\right)_{n\in\N}$ be a sequence that is feasible for $\left(\cl, \func\right)$. The local structure of $\condGraph{A_Q}{\seq}$ is distributed like that of $\unbiased{Q(C_n, q_n)}$, i.e. for each rooted graph $\detGraph$ and $\radius\in\N$ we have
\begin{align*}
	\prob{\ballP{\radius}{\condGraph{A_Q}{\seq}}\isomorphic \detGraph}=\prob{\ballP{\radius}{\unbiased{Q\left(C_n, q_n\right)}}\isomorphic \detGraph}.
\end{align*}
Hence, we obtain $\distributionalLimit{\condGraph{A_Q}{\seq}}{\skeletonTree}$ by \Cref{lem:local_random_complex}. Together with \Cref{lem:conditional} this implies $\distributionalLimit{A_Q}{\skeletonTree}$. Due to \Cref{lem:contiguous2}\ref{lem:contiguous2A} and the fact that \whp\ $\planargraph\in \cl$ we have $\contiguous{\planargraph_Q}{A_Q}$. Thus, we get $\distributionalLimit{\planargraph_Q}{\skeletonTree}$ by \Cref{lem:contiguous1}, as desired.
\end{proof}

\section{Proofs of main results}\label{sec:proofs}
Throughout this section, let $\planargraph=\planargraph(n,m)\ur \planarclass(n,m)$ be the random planar graph.
\proofof{thm:main1}
Let $G=G(n,m)\ur\erClass(n,m)$ be the \ER\ random graph. Due to \Cref{thm:non_complex} we have
\begin{align*}
	\liminf_{n \to \infty}~ \prob{G \text{ is planar }}>0.
\end{align*}
Thus, each graph property that holds \whp\ in $G$, is also true \whp\ in $P$. In particular, \Cref{lem:local_er} remains true if we replace $G$ by $P$, which proves the statement. \qed

\subsection{Proof of \Cref{thm:main2,thm:main3,thm:main4}}\label{sub:main_proofs}
We prove \Cref{thm:main2,thm:main3,thm:main4} simultaneously. Let $\Largestcomponent=\largestcomponent{\planargraph}$ be the largest component of $P$ and $Q=\complexpart{\planargraph}$ the complex part of $\planargraph$. Using the notation from \Cref{def:R_rooted} we can rewrite statements \ref{thm:main2A} and \ref{thm:main2B} to $\distributionalLimit{\planargraph_{\Largestcomponent}}{\skeletonTree}$ and $\distributionalLimit{\planargraph_{\Rest}}{\gwt{1}}$, respectively. Let $\largestcomponent{Q}$ be the largest component of $Q$ and let $\rest{Q}=Q\setminus\largestcomponent{Q}$. Due to \Cref{thm:internal_structure}\ref{thm:internal_structure6} we have \whp\ $\Largestcomponent=\largestcomponent{Q}$. Hence, \whp\ $\symmetricDifference{\vertexSet{\Largestcomponent}}{\vertexSet{Q}}=\vertexSet{\rest{Q}}$. Combining that with \Cref{thm:internal_structure}\ref{thm:internal_structure7} yields $\left|\symmetricDifference{\vertexSet{\Largestcomponent}}{\vertexSet{Q}}\right|=\smallo{\numberVertices{Q}}$. Thus, we obtain by \Cref{lem:contiguous2}\ref{lem:contiguous2B}
\begin{align}\label{eq:4}
\contiguous{\planargraph_\Largestcomponent}{\planargraph_Q}.
\end{align}
Due to \Cref{thm:internal_structure}\ref{thm:internal_structure2} we have \whp\ $\numberVertices{\core{\planargraph}}=\smallo{\numberVertices{\complexpart{\planargraph}}}$ and therefore we can apply \Cref{lem:local_complex} to get $\distributionalLimit{\planargraph_Q}{\skeletonTree}$. Together with (\ref{eq:4}) and \Cref{lem:contiguous1} this yields $\distributionalLimit{\planargraph_L}{\skeletonTree}$, which shows statement \ref{thm:main2A} of \Cref{thm:main2,thm:main3,thm:main4}.

We prove \ref{thm:main2B} similarly. To that end, let $\Restcomplex=\restcomplex{\planargraph}=\planargraph\setminus Q$ be the non--complex part of $\planargraph$ and let $\Rest=\planargraph\setminus\Largestcomponent$. We have \whp\ $\symmetricDifference{\vertexSet{\Rest}}{\vertexSet{\Restcomplex}}=\vertexSet{\rest{Q}}$, as \whp\ $\vertexSet{\Rest}=\vertexSet{\Restcomplex}\cup\vertexSet{\rest{Q}}$ by \Cref{thm:internal_structure}\ref{thm:internal_structure6}. Together with \Cref{thm:internal_structure}\ref{thm:internal_structure8} this implies $\left|\symmetricDifference{\vertexSet{\Rest}}{\vertexSet{\Restcomplex}}\right|=\smallo{\numberVertices{\Restcomplex}}$. Hence, we have $\contiguous{\planargraph_\Rest}{\planargraph_U}$ by \Cref{lem:contiguous2}\ref{lem:contiguous2B}. Next, we combine \Cref{thm:internal_structure}\ref{thm:internal_structure4}, \ref{thm:internal_structure5} and \Cref{lem:local_non_complex} to get $\distributionalLimit{\planargraph_U}{\gwt{1}}$. Together with \Cref{lem:contiguous1} this implies $\distributionalLimit{\planargraph_\Rest}{\gwt{1}}$, which proves \ref{thm:main2B} of \Cref{thm:main2,thm:main3,thm:main4}.

Due to \Cref{thm:internal_structure}\ref{thm:internal_structure3} we have \whp\ $\numberVertices{\Largestcomponent}=\smallo{n}$, $\numberVertices{\Largestcomponent}=\left(c-1+\smallo{1}\right)n$, and $\numberVertices{\Largestcomponent}=\left(1-\smallo{1}\right)n$ if $m$ is as in \Cref{thm:main2}, \Cref{thm:main3}, and \Cref{thm:main4}, respectively. Thus, \ref{thm:main2C} of \Cref{thm:main2,thm:main3,thm:main4} follows by combining \Cref{lem:linear_combination} with statements \ref{thm:main2A} and \ref{thm:main2B}. \qed

\section{Local structure of the core and kernel}\label{sec:local_core_kernel}
Let $\planargraph=\planargraph(n,m)\ur\planarclass(n,m)$ be the random planar graph and we recall that the core $C=\core{\planargraph}$ of $\planargraph$ is the maximal subgraph of the complex part $\complexpart{\planargraph}$ of minimum degree at least two. Furthermore, we obtain the kernel $K=\kernel{\planargraph}$ of $\planargraph$ by replacing all maximal paths in $\core{\planargraph}$ having only internal vertices of degree two by an edge between the endpoints of the path. In this section we consider the local structure of $\planargraph$ around a root which is chosen uniformly at random from $\vertexSet{C}$ and $\vertexSet{K}$, respectively. More precisely, we determine the local limits of the random rooted graphs $\planargraph_C$ and $\planargraph_K$.

If $m\leq n/2+\bigo{n^{2/3}}$, then $\liminf_{n \to \infty}\prob{\vertexSet{C}=\emptyset}=\liminf_{n \to \infty}\prob{\vertexSet{K}=\emptyset}>0$ by \Cref{thm:non_complex}. Hence, we will restrict our considerations to the cases where $m$ is as in \Cref{thm:main2,thm:main3,thm:main4}. We start with the following definition, which provides a generalised version of the Skeleton tree $\skeletonTree$.
\begin{definition}
For $k\in\N_0:=\N\cup\left\{0\right\}$ we denote by $\infinitepath{k}$ the rooted graph consisting of a root and $k$ infinite paths starting at the root, i.e. $\infinitepath{k}$ is the unique rooted tree such that the root has degree $k$ and all other vertices have degree two. Moreover, the {\em Skeleton tree with $k$ rays}, denoted by $\skeletonTreeRays{k}$, is the random rooted graph which is obtained by replacing each vertex of $\infinitepath{k}$ by an independent \GW\ tree $\gwt{1}$ with offspring distribution $\poisson{1}$.  
\end{definition}
\begin{figure}[t]
	\begin{tikzpicture}[scale=1, line width=0.4mm, every node/.style={circle,fill=gray, inner sep=0, minimum size=0.22cm}]
		\node (0) at (0,0) [draw,rectangle]  {};
		\node (11) at (-0.7,0.7) [draw]  {};
		\node (12) at (-1.4,1.4) [draw]  {};
		\node (13) at (-2.1,2.1) [draw]  {};
		\node (14) at (-2.6,2.6) [draw=none, fill=none]  {};
		\node (21) at (0,1) [draw]  {};
		\node (22) at (0,2) [draw]  {};
		\node (23) at (0,3) [draw]  {};
		\node (24) at (0,3.8) [draw=none, fill=none]  {};
		\node (31) at (0.7,0.7) [draw]  {};
		\node (32) at (1.4,1.4) [draw]  {};
		\node (33) at (2.1,2.1) [draw]  {};
		\node (34) at (2.6,2.6) [draw=none, fill=none]  {};		
		\draw[-] (0) to (11);
		\draw[-] (11) to (12);
		\draw[-] (12) to (13);
		\draw[->] (13) to (14);
		\draw[-] (0) to (21);
		\draw[-] (21) to (22);
		\draw[-] (22) to (23);
		\draw[->] (23) to (24);
		\draw[-] (0) to (31);
		\draw[-] (31) to (32);
		\draw[-] (32) to (33);
		\draw[->] (33) to (34);
		
		\node (0a) at (0.8,0) [draw,fill=none]  {};
		\draw[-] (0) to (0a);
		
		\node (31a) at (1.3,0.2) [draw,fill=none]  {};
		\node (31b) at (1.5,0.6) [draw,fill=none]  {};
		\node (31c) at (1.4,1) [draw,fill=none]  {};
		\node (31d) at (2.1,1.4) [draw,fill=none]  {};
		\node (31e) at (2.1,0.9) [draw,fill=none]  {};
		\node (31f) at (2.1,0.5) [draw,fill=none]  {};
		\node (31g) at (2.1,0.1) [draw,fill=none]  {};
		\node (31h) at (2.7,1.6) [draw,fill=none]  {};
		\node (31i) at (2.7,1.2) [draw,fill=none]  {};
		\node (31j) at (2.7,0.1) [draw,fill=none]  {};
		\node (31k) at (3.3,1.5) [draw,fill=none]  {};
		\node (31l) at (3.3,1.2) [draw,fill=none]  {};
		\node (31m) at (3.3,0.9) [draw,fill=none]  {};
		\node (31n) at (3.3,0.5) [draw,fill=none]  {};
		\node (31o) at (3.3,0) [draw,fill=none]  {};
		\node (31p) at (3.9,1.5) [draw,fill=none]  {};
		\node (31q) at (3.9,1.2) [draw,fill=none]  {};
		\node (31r) at (3.9,0.7) [draw,fill=none]  {};
		\node (31s) at (3.9,0.3) [draw,fill=none]  {};
		\node (31t) at (4.5,1.2) [draw,fill=none]  {};
		\node (31u) at (4.5,0.3) [draw,fill=none]  {};
		\node (31v) at (5.1,1.4) [draw,fill=none]  {};
		\node (31w) at (5.1,1) [draw,fill=none]  {};
		\node (31x) at (5.7,1) [draw,fill=none]  {};
		\draw[-] (31) to (31a);
		\draw[-] (31) to (31b);
		\draw[-] (31) to (31c);
		\draw[-] (31a) to (31g);
		\draw[-] (31b) to (31f);
		\draw[-] (31c) to (31d);
		\draw[-] (31c) to (31e);
		\draw[-] (31d) to (31h);
		\draw[-] (31d) to (31i);
		\draw[-] (31g) to (31j);
		\draw[-] (31i) to (31k);
		\draw[-] (31i) to (31l);
		\draw[-] (31i) to (31m);
		\draw[-] (31j) to (31n);
		\draw[-] (31j) to (31o);
		\draw[-] (31k) to (31p);
		\draw[-] (31l) to (31q);
		\draw[-] (31n) to (31r);
		\draw[-] (31n) to (31s);
		\draw[-] (31q) to (31t);
		\draw[-] (31s) to (31u);
		\draw[-] (31t) to (31v);
		\draw[-] (31t) to (31w);
		\draw[-] (31w) to (31x);
		
		\node (33a) at (2.7,2.2) [draw,fill=none]  {};
		\node (33b) at (3.3,2.5) [draw,fill=none]  {};
		\node (33c) at (3.3,2) [draw,fill=none]  {};
		\node (33d) at (3.9,2.5) [draw,fill=none]  {};
		\draw[-] (33) to (33a);	
		\draw[-] (33a) to (33b);	
		\draw[-] (33a) to (33c);	
		\draw[-] (33b) to (33d);
		
		\node (22a) at (0.6,1.9) [draw,fill=none]  {};	
		\node (22b) at (0.6,2.3) [draw,fill=none]  {};
		\node (22c) at (1.2,2.3) [draw,fill=none]  {};
		\draw[-] (22) to (22a);	
		\draw[-] (22) to (22b);	
		\draw[-] (22b) to (22c);
		
		\node (11a) at (-1.3,0.3) [draw,fill=none]  {};
		\node (11b) at (-1.3,0.9) [draw,fill=none]  {};	
		\node (11c) at (-1.9,0.1) [draw,fill=none]  {};	
		\node (11d) at (-1.9,0.45) [draw,fill=none]  {};	
		\node (11e) at (-1.9,0.9) [draw,fill=none]  {};	
		\node (11f) at (-2.5,0.1) [draw,fill=none]  {};	
		\node (11g) at (-2.5,0.9) [draw,fill=none]  {};	
		\draw[-] (11) to (11a);
		\draw[-] (11) to (11b);	
		\draw[-] (11a) to (11c);
		\draw[-] (11a) to (11d);	
		\draw[-] (11b) to (11e);
		\draw[-] (11c) to (11f);
		\draw[-] (11e) to (11g);
		
		\node (13a) at (-2.7,2.1) [draw,fill=none]  {};	
		\node (13b) at (-3.3,2.1) [draw,fill=none]  {};	
		\node (13c) at (-3.9,1.8) [draw,fill=none]  {};	
		\node (13d) at (-3.9,2.4) [draw,fill=none]  {};	
		\node (13e) at (-4.5,2.4) [draw,fill=none]  {};	
		\node (13f) at (-4.5,2) [draw,fill=none]  {};	
		\node (13g) at (-4.5,1.6) [draw,fill=none]  {};	
		\node (13h) at (-5.1,2.4) [draw,fill=none]  {};	
		\node (13i) at (-5.1,1.6) [draw,fill=none]  {};	
		\draw[-] (13) to (13a);
		\draw[-] (13a) to (13b);
		\draw[-] (13b) to (13c);
		\draw[-] (13b) to (13d);
		\draw[-] (13d) to (13e);
		\draw[-] (13c) to (13f);
		\draw[-] (13c) to (13g);
		\draw[-] (13e) to (13h);
		\draw[-] (13g) to (13i);
		
		\node (A) at (-8,0) [draw,rectangle]  {};
		\node (B) at (-8,0.8) [draw]  {};
		\node (C) at (-8,1.6) [draw]  {};
		\node (D) at (-8,2.4) [draw]  {};
		\node (E) at (-8.4,0.6) [draw]  {};
		\node (F) at (-8.8,1.2) [draw]  {};
		\node (G) at (-9.2,1.8) [draw]  {};
		\node (H) at (-7.6,0.6) [draw]  {};
		\node (I) at (-7.2,1.2) [draw]  {};
		\node (J) at (-6.8,1.8) [draw]  {};
		\node (K) at (-8,3.05) [draw=none,fill=none]  {};
		\node (L) at (-9.525,2.2875) [draw=none,fill=none]  {};
		\node (M) at (-6.475,2.2875) [draw=none,fill=none]  {};
		
		\draw[-] (A) to (B);
		\draw[-] (B) to (C);
		\draw[-] (C) to (D);
		\draw[-] (A) to (E);
		\draw[-] (E) to (F);
		\draw[-] (F) to (G);
		\draw[-] (A) to (H);
		\draw[-] (H) to (I);
		\draw[-] (I) to (J);
		\draw[->] (D) to (K);
		\draw[->] (G) to (L);
		\draw[->] (J) to (M);					
	\end{tikzpicture}
	\caption{The rooted graph $\infinitepath{3}$ on the left--hand side and the Skeleton tree $\skeletonTreeRays{3}$ with $k=3$ rays on the right--hand side.}
	\label{fig:skeleton_tree_rays}
\end{figure}
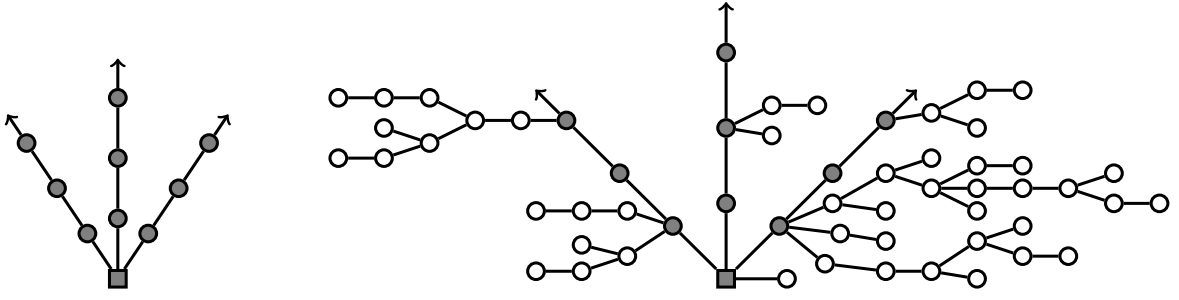
We refer to \Cref{fig:skeleton_tree_rays} for an illustration of $\infinitepath{3}$ and the Skeleton tree $\skeletonTreeRays{3}$ with $k=3$ rays. Now we can state the main result of this section.
\begin{thm}\label{thm:local_core}
Let $\planargraph=\planargraph(n,m)\ur\planarclass(n,m)$, $C=\core{\planargraph}$ be the core of $\planargraph$, and $K=\kernel{\planargraph}$ the kernel. Suppose that $m=m(n)$ lies in one of the three regimes where $m=n/2+s$ with $s=s(n)>0$ satisfying $s=o(n)$ and $s^3n^{-2}\to\infty$, $m=\alpha n/2$ with $\alpha=\alpha(n)$ tending to a constant in $(1,2)$, or $m=n+\smallo{n}$ with $m\leq n+\smallo{n\left(\log n\right)^{-2/3}}$. Then $\distributionalLimit{\planargraph_C}{\skeletonTreeRays{2}}$ and $\distributionalLimit{\planargraph_K}{\skeletonTreeRays{3}}$.
\end{thm}
Before providing a proof of \Cref{thm:local_core}, we relate this result to the local limits of the non--complex part $\Restcomplex=\restcomplex{\planargraph}$ of $\planargraph$ and the complex part $Q=\complexpart{\planargraph}$. In \Cref{sub:main_proofs} we showed $\distributionalLimit{\planargraph_U}{\gwt{1}}$ and  $\distributionalLimit{\planargraph_Q}{\skeletonTree}$. By definition, $\gwt{1}$ is distributed like $\skeletonTreeRays{0}$ and the Skeleton tree $\skeletonTree$ like $\skeletonTreeRays{1}$. Thus, we have the following \lq sequence\rq\ of local limits
\begin{align*}
\distributionalLimit{\planargraph_U}{\skeletonTreeRays{0}};\\
\distributionalLimit{\planargraph_Q}{\skeletonTreeRays{1}};\\
\distributionalLimit{\planargraph_C}{\skeletonTreeRays{2}};\\ \distributionalLimit{\planargraph_K}{\skeletonTreeRays{3}}.
\end{align*}

In the remaining part of this section we prove \Cref{thm:local_core}. We will restrict our considerations to the complex part $Q$, as the local limits of $\planargraph_C$ and $\planargraph_K$ are completely determined by $Q$. We recall that the complex part $Q$ can be obtained by replacing each vertex of the core $C$ by a rooted tree. This motivates to proceed in two proof steps, which mirror the definition of $\skeletonTreeRays{k}$ via $\infinitepath{k}$. More precisely, we will first consider only the core $C$ and will show that $\distributionalLimit{\unbiased{ C}}{\infinitepath{2}}$ and $\distributionalLimit{C_K}{\infinitepath{3}}$ (see \Cref{lem:core_local}). Then in the second step we will prove that the trees attached to the vertices in $C$ to obtain $Q$ behave asymptotically like independent \GW\ trees with common offspring distribution $\poisson{1}$ (see \Cref{lem:limit_trees_forest}). 

Instead of directly analysing the local structure of $C$, we will first determine the local structure of a random core with a fixed kernel and a fixed number of vertices. To that end, we introduce the following definition. 

\begin{definition}\label{def:random_core}
	Let $K$ be a kernel and $k\in\N$. The random core $C(K,k)$ is a graph chosen uniformly at random from the class of all cores having $K$ as its kernel and $\numberVertices{K}+k$ vertices. 
\end{definition}

We note that $C(K,k)$ can be obtained by randomly subdividing the edges of $K$ with $k$ additional vertices. In our applications the number of edges in $K$ will be much smaller than $k$ (see \Cref{thm:internal_structure}\ref{thm:internal_structure1}). Therefore, we expect that an edge $e\in\edgeSet{K}$ will be subdivided by \lq many\rq\ vertices. In the following we make that more precise by introducing the subdivision number of an edge and showing that it is \lq large\rq.

\begin{definition}
Let $C$ be a given core with kernel $K=\kernel{C}$. To obtain $C$ from $K$ we have to subdivide each edge of $K$ with a certain number of vertices. For a fixed edge $e\in\edgeSet{K}$ we call this number of vertices the {\em subdivision number} $\subdivisionNumber{C}{e}$.
\end{definition}

\begin{lem}\label{lem:subdivision_number}
For each $n\in\N$, let $K=K(n)$ be a kernel, $k=k(n)\in\N$, $e=e(n)\in\edgeSet{K}$, and $C=C(n)=C(K,k)$ the random core as defined in \Cref{def:random_core}. If $\numberEdges{K}=\smallo{k}$, then \whp\ $\subdivisionNumber{C}{e}=\smallomega{1}$.
\end{lem}
\begin{proof}
We denote by $\mathcal{C}$ the class of all cores having $K$ as its kernel and $\numberVertices{K}+k$ vertices, i.e. $C(K,k)$ is chosen uniformly at random from $\mathcal{C}$. Furthermore, for $i\in\N_0$ let $\mathcal{C}_i\subseteq\mathcal{C}$ be the subclass containing all cores $H\in \mathcal{C}$ for which $\subdivisionNumber{H}{e}=i$. First we aim to find an upper bound for $\left|\mathcal{C}_i\right|/\left|\mathcal{C}_{i+1}\right|$ using a double counting argument. To that end, let $i\in\N_0$ and $H\in\mathcal{C}_i$ be given. Then we can build a graph $H'\in\mathcal{C}_{i+1}$ via the following construction. We start by picking a vertex $v\in\vertexSet{H}\setminus\vertexSet{K}$ which does not lie on $e$, but on an edge $f\in\edgeSet{K}$ with $\subdivisionNumber{H}{f}\geq 3$. Then we choose an edge $e'\in\edgeSet{H}$ which is a \lq part\rq\ of $e$ and move the vertex $v$ to the edge $e'$ (see also \Cref{fig:double_counting}). We note that we have at least $\left(k-i-2\numberEdges{K}\right)$ different choices for $v$ and precisely $(i+1)$ for $e'$. Considering the reverse operation we obtain that for a fixed graph $H'\in\mathcal{C}_{i+1}$ there are at most $(i+1)\left(\numberEdges{K}+k-i-2\right)$ graphs $H\in\mathcal{C}_i$ which can be transformed to $H'$ via the above operation. Thus, we obtain
\begin{align*}
\frac{\left|\mathcal{C}_i\right|}{\left|\mathcal{C}_{i+1}\right|}\leq \frac{(i+1)\left(\numberEdges{K}+k-i-2\right)}{\left(k-i-2\numberEdges{K}\right)(i+1)}\leq 1+\frac{3\numberEdges{K}}{k-i-2\numberEdges{K}}\leq\exp\left(\frac{3\numberEdges{K}}{k-i-2\numberEdges{K}}\right).
\end{align*}
Combining that with the assumption $\numberEdges{K}=\smallo{k}$ we obtain for $n$ large enough and all $i\leq N=N(n):=\lfloor k/\numberEdges{K}\rfloor$
\begin{align}\label{eq:22}
\frac{\left|\mathcal{C}_i\right|}{\left|\mathcal{C}_{i+1}\right|}\leq\exp\left(\frac{8\numberEdges{K}}{k}\right).
\end{align}
For fixed $j\in\N_0$ inequality (\ref{eq:22}) implies that for all $i\in\N$ with $j\leq i\leq N$ we have
\begin{align}\label{eq:23}
\frac{\left|\mathcal{C}_j\right|}{\left|\mathcal{C}_i\right|}=\prod_{i'=j}^{i-1}\frac{\left|\mathcal{C}_{i'}\right|}{\left|\mathcal{C}_{i'+1}\right|}\leq \exp\left(\frac{8\numberEdges{K}(i-j)}{k}\right)\leq  \exp\left(\frac{8\numberEdges{K}N}{k}\right) \leq e^8,
\end{align}
where we used in the last inequality that $N=\lfloor k/\numberEdges{K}\rfloor$.
Summing up (\ref{eq:23}) for all $i\in\N$ with $j\leq i\leq N$ yields  
\begin{align*}
	(N-j+1)\left|\mathcal{C}_j\right| \leq e^8\sum_{i=j}^{N}\left|\mathcal{C}_i\right|=\bigo{\left|\mathcal{C}\right|},
\end{align*}
as $\sum_{i=j}^{N}\left|\mathcal{C}_i\right|\leq \left|\mathcal{C}\right|$. Combining that with the fact $N-j+1=\smallomega{1}$, we obtain $\prob{\subdivisionNumber{C}{e}=j}=\left|\mathcal{C}_j\right|/\left|\mathcal{C}\right|=\smallo{1}$. Thus, we have $\prob{\subdivisionNumber{C}{e}\leq j}=\smallo{1}$ for each fixed $j\in \N_0$, which shows the statement.
\end{proof}
\begin{figure}[t]
	\begin{tikzpicture}[scale=1, line width=0.4mm, every node/.style={circle, fill=gray, inner sep=0, minimum size=0.22cm}, decoration=brace]
		\clip (-0.3cm,-0.7cm) rectangle (15.5cm,2.2cm);
		\node (1) at (0,0) [draw,rectangle]  {};
		\node (2) at (3,0) [draw,rectangle]  {};
		\node (3) at (4.8,0) [draw,rectangle]  {};
		\node (4) at (1.5,2) [draw,rectangle]  {};
		\node (5) at (1.5,0) [draw]  {};
		\node (6) at (0.5,2/3) [draw]  {};
		\node (7) at (1,4/3) [draw]  {};
		\node (8) at (2.25,1) [draw]  {};
		\node (9) at (1.875,1.5) [draw,fill=none]  {};
		\node (9a) at (1.6,1.5) [draw=none,fill=none]  {$v$};
		\node (10) at (2.625,0.5) [draw]  {};
		\node (11) at (3.9,0) [draw]  {};
		
		\draw[-] (1) to (5);
		\draw[-] (2) to (5);
		\draw[-] (1) to (6);
		\draw[-] (6) to (7);
		\draw[-] (4) to (7);
		\draw[-] (2) to (10);
		\draw[-] (4) to (9);
		\draw[-] (8) to (10);
		\draw[-] (8) to (9);
		\draw[-] (2) to (11);
		\draw[-] (3) to (11);	
		\draw [out=100,in=180,looseness=1](1) to  (4);
		\node (12) at (-0.02,0.7) [draw]  {};
		\node (13) at (0.225,1.325) [draw]  {};
		\node (14) at (0.72,1.82) [draw]  {};
		\draw [out=0,in=90,looseness=50](3) to  (3);
		\node (15) at (5.8,0.11) [draw]  {};	
		\node (16) at (6.5,0.8) [draw]  {};
		\node (17) at (5.8,1.65) [draw]  {};
		\node (18) at (4.9,1) [draw]  {};
		\draw[decorate, yshift=-1.5ex] (3,0) -- node[below=1ex,draw=none,fill=none] {$e$} (0,0);
		\draw[decorate, yshift=1ex]  (1.5,0) -- node[above=0.4ex,draw=none,fill=none] {$e'$}  (3,0);
		\draw[decorate, yshift=1ex, xshift=1ex]  (1.5,2) -- node[right=0.6ex, draw=none,fill=none] {$f$}  (3,0);
		\node (19) at (7.6,0.9) [circle, fill=none]  {\Huge $\rightarrow$};

		\node (1A) at (8.8,0) [draw,rectangle]  {};
		\node (2A) at (11.8,0) [draw,rectangle]  {};
		\node (3A) at (13.6,0) [draw,rectangle]  {};
		\node (4A) at (10.3,2) [draw,rectangle]  {};
		\node (5A) at (10.3,0) [draw]  {};
		\node (6A) at (9.3,2/3) [draw]  {};
		\node (7A) at (9.8,4/3) [draw]  {};
		\node (8A) at (11.05,1) [draw]  {};
		\node (9A) at (11.05,0) [draw,fill=none]  {};
		\node (9aA) at (11.05,-0.3) [draw=none,fill=none]  {$v$};
		\node (10A) at (11.425,0.5) [draw]  {};
		\node (11A) at (12.7,0) [draw]  {};
		
		\draw[-] (1A) to (5A);
		\draw[-] (5A) to (9A);
		\draw[-] (1A) to (6A);
		\draw[-] (6A) to (7A);
		\draw[-] (4A) to (7A);
		\draw[-] (2A) to (10A);
		\draw[-] (4A) to (8A);
		\draw[-] (8A) to (10A);
		\draw[-] (2A) to (9A);
		\draw[-] (2A) to (11A);
		\draw[-] (3A) to (11A);	
		\draw [out=100,in=180,looseness=1](1A) to  (4A);
		\node (12A) at (8.78,0.7) [draw]  {};
		\node (13A) at (9.025,1.325) [draw]  {};
		\node (14A) at (9.52,1.82) [draw]  {};
		\draw [out=0,in=90,looseness=50](3A) to  (3A);
		\node (15A) at (14.6,0.11) [draw]  {};	
		\node (16A) at (15.3,0.8) [draw]  {};
		\node (17A) at (14.6,1.65) [draw]  {};
		\node (18A) at (13.7,1) [draw]  {};	
	\end{tikzpicture}
	\caption{The construction used in the double counting argument in the proof of \Cref{lem:subdivision_number}. The vertices lying in the kernel are marked by a square, the others by a circle.}
	\label{fig:double_counting}
\end{figure}
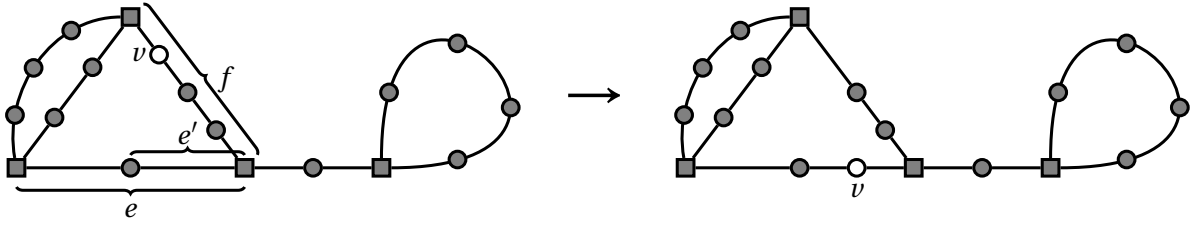
In our applications a vertex $v\in\vertexSet{K}$ will typically have degree three in $K$ (see \Cref{thm:internal_structure}\ref{thm:internal_structure9}). Then in the process of obtaining $C(K,k)$ the three edges of $K$ incident to $v$ will be subdivided by \lq many\rq\ vertices due to \Cref{lem:subdivision_number}. Hence, the local limit of $C(K,k)$ around $v$ should be $\infinitepath{3}$. Moreover, $K$ will typically be quite \lq small\rq\ compared to $k$ (see \Cref{thm:internal_structure}\ref{thm:internal_structure1}). Thus, \lq most\rq\ of the vertices in $C(K,k)$ have degree two, which will imply that the local limit of $C(K,k)$ around a randomly chosen vertex is $\infinitepath{2}$.
\begin{lem}\label{lem:random_core_local}
	For each $n\in\N$, let $K=K(n)$ be a kernel, $k=k(n)\in\N$, and $C=C(K,k)$ the random core as defined in \Cref{def:random_core}. If $\numberEdges{K}=\smallo{k}$ and $\numberEdges{K}=\left(3/2+\smallo{1}\right)\numberVertices{K}$, then $\distributionalLimit{\unbiased{C}}{\infinitepath{2}}$ and $\distributionalLimit{C_K}{\infinitepath{3}}$.
\end{lem}
\begin{proof}
We assume that $\unbiased{C}=\rootedGraph{C}{r_C}$, i.e. $r_C\ur\vertexSet{C}$, and let $\radius\in\N$. We note that the core $C$ arises from the kernel $K$ by subdividing edges. Thus, there are at most $2\radius\numberEdges{K}$ vertices $v\in\vertexSet{C}\setminus\vertexSet{K}$ such that $\vertexSet{\ball{\radius}{C}{v}}\cap\vertexSet{K} \neq\emptyset$. Hence,
\begin{align*}
	\prob{\vertexSet{\ball{\radius}{C}{r_C}}\cap\vertexSet{K}\neq \emptyset}\leq\frac{\numberVertices{K}+2\radius\numberEdges{K}}{\numberVertices{C}}=\smallo{1},
\end{align*}
as $\numberEdges{K}=\smallo{\numberVertices{C}}$ and $\numberEdges{K}=\left(3/2+\smallo{1}\right)\numberVertices{K}$.
This shows $\distributionalLimit{\unbiased{C}}{\infinitepath{2}}$.

Next, we suppose that $C_K=\rootedGraph{C}{r_K}$, i.e. $r_K\ur\vertexSet{K}$. As each realisation of $C$ has $K$ as its kernel, we can assume that we first pick $r_K$ and then independently choose the realisation of $C$. We note that \whp\ the degree of $r_K$ in $K$ is three, since $\numberEdges{K}=\left(3/2+\smallo{1}\right)\numberVertices{K}$ and the minimum degree of $K$ is at least three. Furthermore, for each edge $e\in \edgeSet{K}$ we have \whp\ $\subdivisionNumber{C}{e}=\smallomega{1}$ by \Cref{lem:subdivision_number}. In particular, this is true for all edges that are incident to $r_K$. Hence, we obtain $\distributionalLimit{C_K}{\infinitepath{3}}$, which finishes the proof.
\end{proof}

Next, we will transfer the results of \Cref{lem:random_core_local} to the core $\core{\planargraph}$ of the random planar graph $\planargraph$ by using the concept of conditional random rooted graphs and \Cref{lem:conditional}.
\begin{lem}\label{lem:core_local}
	Let $\planargraph=\planargraph(n,m)\ur\planarclass(n,m)$, $C=\core{\planargraph}$ be the core of $\planargraph$, and $K=\kernel{\planargraph}$ the kernel. Suppose that $m=m(n)$ lies in one of the three regimes of \Cref{thm:main2,thm:main3,thm:main4}. Then $\distributionalLimit{\unbiased{C}}{\infinitepath{2}}$ and $\distributionalLimit{C_K}{\infinitepath{3}}$.
\end{lem}
\begin{proof}
We will use \Cref{lem:conditional}. To that end, let $\cl(n)$ be the subclass of $\planarclass(n,m)$ containing those graphs $H$ satisfying $\numberEdges{\kernel{H}}=\smallo{\numberVertices{\core{H}}}$ and $\numberEdges{\kernel{H}}=\left(3/2+\smallo{1}\right)\numberVertices{\kernel{H}}$. By \Cref{thm:internal_structure}\ref{thm:internal_structure1} and \ref{thm:internal_structure9} we have \whp\ $\planargraph\in \cl:=\cup_n\cl(n)$. Let $A=A(n)\ur\cl(n)$ and we define the function $\func$ such that $\func(H)=\big(\kernel{H}, \numberVertices{\core{H}}-\numberVertices{\kernel{H}}\big)$ for each $H\in\cl$. Let $\seq=\left(K_n, k_n\right)_{n\in\N}$ be a sequence that is feasible for $\left(\cl, \func\right)$. Moreover, let $\core{A}_C$ and $\core{A}_K$ be the random rooted graphs where we first pick the realisation of $\core{A}$ and then independently choose the root uniformly at random from $\vertexSet{\core{A}}$ and $\vertexSet{\kernel{A}}$, respectively. The local structures of $\condGraph{\core{A}_C}{\seq}$ and $\condGraph{\core{A}_K}{\seq}$ are distributed like that of $\unbiased{C(K_n, k_n)}$ and $C(K_n, k_n)_K$, respectively. Thus, we obtain by \Cref{lem:random_core_local,lem:conditional} that $\distributionalLimit{\unbiased{\core{A}}}{\infinitepath{2}}$ and $\distributionalLimit{\core{A}_K}{\infinitepath{3}}$. Using \Cref{lem:contiguous2}\ref{lem:contiguous2A} and the fact that \whp\ $\planargraph\in \cl$ we obtain $\contiguous{\unbiased{C}}{\unbiased{\core{A}}}$ and $\contiguous{C_K}{\core{A}_K}$. Thus, the statement follows by \Cref{lem:contiguous1}.
\end{proof}

In the second main part of the proof of \Cref{thm:local_core} we consider the trees which we attach to the vertices of the core $\core{\planargraph}$ to obtain the complex part $\complexpart{\planargraph}$. More precisely, we analyse the components of the random forest $\forest(n,\ntrees)$, which is a graph chosen uniformly at random from the class $\forestClass(n,\ntrees)$ of all forests on vertex set $[n]$ having $\ntrees$ trees as components such that the vertices $1, \ldots, \ntrees$ lie all in different tree components (cf. \Cref{def:random_forest}). We will show that if we consider only a fixed number of tree components of $\forest(n,\ntrees)$, then these trees behave asymptotically like independent \GW\ trees with common offspring distribution $\poisson{1}$ (see \Cref{lem:limit_trees_forest}). To that end, we start by proving that these trees have the \lq right\rq\ number of vertices, which is given by \Cref{thm:total_progeny}.
\begin{lem}\label{lem:number_vertices_gwt}
Let $\ntrees=\ntrees(n)=\smallo{n}$ be such that $t=\smallomega{1}$, $N\in\N$, $k_1, \ldots, k_N\in\N$, and $\forest=\forest(n,\ntrees)\ur\forestClass(n,\ntrees)$. For $i\in[\ntrees]$ we denote by $F_i$ the tree component of $\forest$ containing vertex $i$. Then we have
\begin{align*}
\prob{\forall i\in[N]:\numberVertices{F_i}=k_i }=\left(1+\smallo{1}\right)\prod_{i=1}^{N}e^{-k_i}k_i^{k_i-1}/k_i!.
\end{align*}
\end{lem}
\begin{proof}
Every realisation of $\forest$ such that $\numberVertices{F_i}=k_i$ for each $i\in[N]$ can be constructed as follows. First we choose pairwise disjoint subsets $S_1, \ldots, S_N\subseteq [n]\setminus[t]$ such that $\left|S_i\right|=k_i-1$ for all $i\in[N]$. Then we pick a tree on vertex set $S_i\cup \left\{i\right\}$ for each $i\in[N]$. Finally, we add a forest on vertex set $[n]\setminus\left(\bigcup_{i\in[N]}S_i\cup\left\{i\right\}\right)$ with $(\ntrees-N)$ tree components such that the vertices $N+1, \ldots, t$ lie all in different tree components. Thus, we obtain with $K_i:=\sum_{j=1}^{i}k_j$
\begin{align*}
\prob{\forall i\in[N]:\numberVertices{F_i}=k_i}&=\left(\prod_{i=1}^{N}\binom{n-\ntrees-K_{i-1}}{k_i-1}k_i^{k_i-2}\right)\cdot (\ntrees-N)\left(n-K_N\right)^{n-K_N-\ntrees+N-1}/\left(\ntrees n^{n-t-1}\right)
\\
&=\left(1+\smallo{1}\right)\left(\prod_{i=1}^{N}n^{k_i-1}/\left(k_i-1\right)!\cdot  k_i^{k_i-2}\right)n^{n-K_N-\ntrees+N-1}e^{-K_N}/n^{n-\ntrees-1}
\\
&=\left(1+\smallo{1}\right)\prod_{i=1}^{N}e^{-k_i}k_i^{k_i-1}/k_i!,
\end{align*}
as desired.
\end{proof}
In the proof of \Cref{lem:limit_trees_forest} we will use the following two known facts.
\begin{thm}[{\cite[Theorem 3.16]{vdH}}]\label{thm:total_progeny}
Let $\gwt{1}$ be a \GW\ tree with offspring distribution $\poisson{1}$. Then for each $k\in\N$
\begin{align*}
	\prob{\numberVertices{\gwt{1}}=k}=e^{-k}k^{k-1}/k!.
\end{align*}
\end{thm}

\begin{thm}[\cite{continuum_random_tree}]\label{thm:cayley_gwt_tree}
Let $k\in\N$ be fixed, $T=T(k)$ the random tree on vertex set $[k]$, and $\gwt{1}$ the \GW\ tree with offspring distribution $\poisson{1}$. Then for each rooted graph $\detGraph$
\begin{align*}
\prob{\unbiased{T}\isomorphic\detGraph}=\condprob{\gwt{1}\isomorphic\detGraph}{\numberVertices{\gwt{1}}=k}.
\end{align*} 
\end{thm}
\Cref{thm:cayley_gwt_tree} says that $T(k)$, considered as a rooted graph, is distributed like a \GW\ tree $\gwt{1}$ conditioned on the event that $\gwt{1}$ has $k$ vertices. Now we can determine the asymptotic behaviour of the components of the random forest $F(n,\ntrees)$.
\begin{lem}\label{lem:limit_trees_forest}
Let $\ntrees=\ntrees(n)=\smallo{n}$ be such that $t=\smallomega{1}$, $N\in\N$, $\detGraph_1, \ldots, \detGraph_N$ rooted graphs, and $\forest=\forest(n,\ntrees)\ur\forestClass(n,\ntrees)$. For $i\in[\ntrees]$ we denote by $F_i$ the tree component of $\forest$ containing vertex $i$. Then we have
\begin{align*}
	\prob{\forall i\in[N]:\rootedGraph{F_i}{i}\isomorphic \detGraph_i }=\left(1+\smallo{1}\right)\prod_{i=1}^{N}\prob{\gwt{1}\isomorphic\detGraph_i}.
\end{align*}
\end{lem}
\begin{proof}
Let $k_1, \ldots, k_N$ be the number of vertices in the rooted graphs $\detGraph_1, \ldots, \detGraph_N$, respectively. Using \Cref{lem:number_vertices_gwt} we obtain
\begin{align}\label{eq:20}
\prob{\forall i\in[N]:\rootedGraph{F_i}{i}\isomorphic \detGraph_i}&=\prob{\forall i\in[N]:\numberVertices{F_i}=k_i}~\condprob{\forall i\in[N]:\rootedGraph{F_i}{i}\isomorphic \detGraph_i}{\forall i\in[N]:\numberVertices{F_i}=k_i}\nonumber
\\
&=\left(1+\smallo{1}\right)\left(\prod_{i=1}^{N}e^{-k_i}k_i^{k_i-1}/k_i!\right)~\condprob{\forall i\in[N]:\rootedGraph{F_i}{i}\isomorphic \detGraph_i}{\forall i\in[N]:\numberVertices{F_i}=k_i}.
\end{align}
Conditioned on the event $\left\{\forall i\in[N]:\numberVertices{F_i}=k_i\right\}$ the tree components $F_1, \ldots, F_N$ are independent and distributed like random trees $T(k_1), \ldots, T(k_N)$, respectively. Thus, we obtain by \Cref{thm:cayley_gwt_tree}
\begin{align}\label{eq:19}
\condprob{\forall i\in[N]:\rootedGraph{F_i}{i}\isomorphic \detGraph_i}{\forall i\in[N]:\numberVertices{F_i}=k_i}&=\prod_{i=1}^{N}\prob{\unbiased{T(k_i)}\isomorphic \detGraph_i}\nonumber
\\
&=\prod_{i=1}^{N}\condprob{\gwt{1}\isomorphic\detGraph_i}{\numberVertices{\gwt{1}}=k_i}\nonumber
\\
&=\prod_{i=1}^{N}\prob{\gwt{1}\isomorphic\detGraph_i}/\prob{\numberVertices{\gwt{1}}=k_i}.
\end{align}
Using \Cref{thm:total_progeny} in (\ref{eq:19}) we obtain
\begin{align*}
\condprob{\forall i\in[N]:\rootedGraph{F_i}{i}\isomorphic \detGraph_i}{\forall i\in[N]:\numberVertices{F_i}=k_i}=\prod_{i=1}^{N}\prob{\gwt{1}\isomorphic\detGraph_i}e^{k_i}k_i^{-k_i+1}k_i!.
\end{align*}
Plugging in that in (\ref{eq:20}) yields
\begin{align*}
\prob{\forall i\in[N]:\rootedGraph{F_i}{i}\isomorphic \detGraph_i}&=\left(1+\smallo{1}\right)\left(\prod_{i=1}^{N}e^{-k_i}k_i^{k_i-1}/k_i!\right)~\left(\prod_{i=1}^{N}\prob{\gwt{1}\isomorphic\detGraph_i}e^{k_i}k_i^{-k_i+1}k_i!\right)\\
&=\left(1+\smallo{1}\right)\prod_{i=1}^{N}\prob{\gwt{1}\isomorphic\detGraph_i},
\end{align*}
which finishes the proof.
\end{proof}

In \Cref{lem:core_local} we obtained the local structure of the core $\core{\planargraph}$ and in \Cref{lem:limit_trees_forest} the asymptotic behaviour of the trees which we attach to $\core{\planargraph}$. Combining these two statements we can prove \Cref{thm:local_core} by first considering the random complex part $Q(C,q)$ as defined in \Cref{def:random_complex} and then using the concept of conditional random rooted graphs.
\begin{lem}\label{lem:local_core_complex}
	For each $n\in\N$, let $C=C(n)$ be a core with kernel $K=\kernel{C}$, $q=q(n)\in\N$, and $Q=Q(C,q)$ be the random complex part with core $C$ and vertex set $[q]$. Suppose that $\numberEdges{K}=\smallo{\numberVertices{C}}$, $\numberVertices{C}=\smallo{q}$, $\distributionalLimit{\unbiased{C}}{\infinitepath{2}}$, and $\distributionalLimit{C_K}{\infinitepath{3}}$. Then we have $\distributionalLimit{Q_C}{\skeletonTreeRays{2}}$ and $\distributionalLimit{Q_K}{\skeletonTreeRays{3}}$.
\end{lem}
\begin{proof}
We assume that $Q_C=\rootedGraph{Q}{r_C}$, i.e. $r_C\ur\vertexSet{C}$, and let $\radius\in\N$. We recall that $Q$ can be constructed by choosing a random forest $\forest=\forest(q, \numberVertices{C})$ and replacing each vertex $v$ in $C$ by the tree component of $\forest$ rooted at $v$. We can assume that we first pick the root $\root_C$ and then independently the forest $F$. Then $\ball{\radius}{Q}{\root_C}$ depends only on those trees of $F$ which are attached to vertices in $\ball{\radius}{C}{\root_C}$. Due to the assumption $\distributionalLimit{Q_C}{\skeletonTreeRays{2}}$ \whp\ $\ball{\radius}{C}{\root_C}$ is a graph consisting of two paths of length $\radius$ starting at $\root_C$. If this is the case, then $\numberVertices{\ball{\radius}{C}{\root_C}}=2\radius+1$ and we can apply \Cref{lem:limit_trees_forest} to the $2\radius+1$ trees attached to the vertices in 
$\vertexSet{\ball{\radius}{C}{\root_C}}$, since $\numberVertices{C}=\smallo{q}$ and the assumption $\numberEdges{\kernel{C}}=\smallo{\numberVertices{C}}$ implies $\numberVertices{C}=\smallomega{1}$. Hence, these trees behave asymptotically like independent \GW\ trees $\gwt{1}$ with offspring distribution $\poisson{1}$, which implies $\distributionalLimit{Q_C}{\skeletonTreeRays{2}}$. The second statement $\distributionalLimit{Q_K}{\skeletonTreeRays{3}}$ follows analogously.
\end{proof}

\proofofW{thm:local_core}
In order to use \Cref{lem:conditional}, let $\cl(n)$ be the subclass of $\planarclass(n,m)$ containing those graphs $H$ satisfying $\numberEdges{\kernel{H}}=\smallo{\numberVertices{\core{H}}}$, $\numberVertices{\core{H}}=\smallo{\numberVertices{\complexpart{H}}}$, and $\distributionalLimit{\unbiased{\core{H}}}{\infinitepath{2}}$. Due to \Cref{thm:internal_structure}\ref{thm:internal_structure1}, \ref{thm:internal_structure2} and \Cref{lem:core_local} we have \whp\ $\planargraph\in \cl:=\cup_n\cl(n)$. Let $A_C=A_C(n)$ be the random rooted graph where we first pick $A=A(n)\ur\cl(n)$ and given the realisation $H$ of $A$, the root is chosen uniformly at random from the vertex set of the core $\core{H}$. We define the function $\func$ such that $\func(H)=\left(\core{H}, \numberVertices{\complexpart{H}}\right)$ for each $H\in\cl$ and let $\seq=\left(C_n, q_n\right)_{n\in\N}$ be a sequence that is feasible for $\left(\cl, \func\right)$. The local structure of $\condGraph{A_C}{\seq}$ is distributed like that of $Q(C_n, q_n)_{C_n}$, which is the graph $Q(C_n, q_n)$, as defined in \Cref{def:random_complex}, rooted at a randomly chosen vertex from the core $C_n$. Hence, we obtain $\distributionalLimit{\condGraph{A_C}{\seq}}{\skeletonTreeRays{2}}$ by \Cref{lem:local_core_complex}, which implies  $\distributionalLimit{A_C}{\skeletonTreeRays{2}}$ due to \Cref{lem:conditional}. Using \Cref{lem:contiguous2}\ref{lem:contiguous2A} and the fact that \whp\ $\planargraph\in \cl$ we obtain $\contiguous{\planargraph_C}{A_C}$. Thus, we have $\distributionalLimit{\planargraph_C}{\skeletonTreeRays{2}}$ by \Cref{lem:contiguous1}. The assertion $\distributionalLimit{\planargraph_K}{\skeletonTreeRays{3}}$ can be shown analogously. \qed

\section{Discussion}\label{sec:discussion}
The only properties of $\planargraph(n,m)$ used in our proofs are the statements on the internal structure from \Cref{thm:internal_structure}. Kang, Moßhammer, and Sprüssel \cite{surface} proved that \Cref{thm:internal_structure} holds also for many other graph classes, including cactus graphs, series--parallel graphs, and graphs embeddable on an orientable surface of genus $g\in\N_0$ (see also \cite[Section 4]{cycles}). Combining the generalised version of \Cref{thm:internal_structure} and analogous proofs of \Cref{thm:main1,thm:main2,thm:main3,thm:main4,thm:local_core}, one can show the following result, which answers a question of Kang, Moßhammer, and Sprüssel \cite[Question 8.7]{surface}.
\begin{thm}
	\Cref{thm:main1,thm:main2,thm:main3,thm:main4,thm:local_core} are also true for the class of cactus graphs, the class of series--parallel graphs, and the class of graphs embeddable on an orientable surface of genus $g\in \N_0$.
\end{thm}

\bibliographystyle{plain}
\bibliography{kang-missethan-local-planar}
\newpage
\appendix
\section{Proof of \Cref{lem:local_er}}\label{app:local_er}
In order to show \Cref{lem:local_er} we will consider plane trees, i.e. rooted unlabelled trees in which the children of each vertex are distinguishable. For instance, we can think that such a tree is embedded in the plane and the children of a vertex are ordered from \lq left\rq\ to \lq right\rq. We note that a rooted labelled tree $\rootedGraph{H}{\root}$ corresponds to a plane tree $\planeTree$ in a canonical way: The root of $\planeTree$ is $\root$ and the children of a vertex are ordered according to their labels in $\rootedGraph{H}{\root}$, e.g. in ascending order. Finally, we \lq forget\rq\ the labels in $H$ to obtain $\planeTree$. We refer to \Cref{fig:plane_tree} for an illustration of connecting a rooted labelled tree to a plane tree. Given a plane tree $\planeTree$, a rooted graph $\rootedGraph{H}{\root}$, and $\radius\in\N$, we will write in the following $\ball{\radius}{H}{\root}=\planeTree$ to express that $\ball{\radius}{H}{\root}$ is a tree and the corresponding plane tree equals $\planeTree$. 

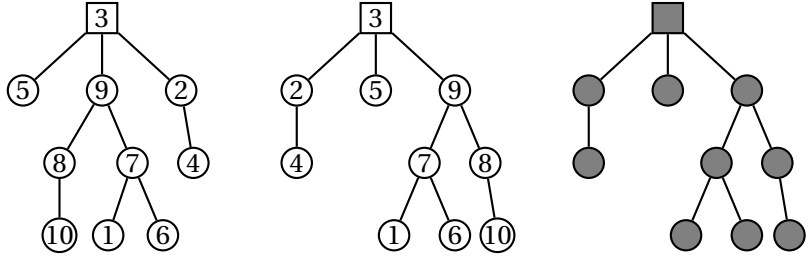
\begin{figure}[t]
	\centering
	\begin{tikzpicture}[thick, every node/.style={circle, inner sep=0, minimum size=0.4cm}, scale=0.8]
		
		\node (A) at (0,0) [draw,rectangle]  {3};
		\node (B) at (-1.3,-1.2) [draw]  {5};
		\node (C) at (0,-1.2) [draw]  {9};
		\node (D) at (1.3,-1.2) [draw]  {2};
		\node (E) at (-0.7,-2.4) [draw]  {8};
		\node (F) at (0.5,-2.4) [draw]  {7};
		\node (G) at (1.5,-2.4) [draw]  {4};
		\node (H) at (-0.7,-3.6) [draw]  {10};
		\node (I) at (0.1,-3.6) [draw]  {1};
		\node (J) at (1,-3.6) [draw]  {6};
		\draw[-] (A) to (B);
		\draw[-] (A) to (C);
		\draw[-] (A) to (D);
		\draw[-] (C) to (E);
		\draw[-] (C) to (F);
		\draw[-] (D) to (G);
		\draw[-] (E) to (H);
		\draw[-] (F) to (I);
		\draw[-] (F) to (J);
				
		\node (A1) at (4.5,0) [draw,rectangle]  {3};
		\node (B1) at (4.5,-1.2) [draw]  {5};
		\node (C1) at (5.8,-1.2) [draw]  {9};
		\node (D1) at (3.2,-1.2) [draw]  {2};
		\node (E1) at (6.3,-2.4) [draw]  {8};
		\node (F1) at (5.3,-2.4) [draw]  {7};
		\node (G1) at (3.2,-2.4) [draw]  {4};
		\node (H1) at (6.5,-3.6) [draw]  {10};
		\node (I1) at (4.8,-3.6) [draw]  {1};
		\node (J1) at (5.8,-3.6) [draw]  {6};
		\draw[-] (A1) to (B1);
		\draw[-] (A1) to (C1);
		\draw[-] (A1) to (D1);
		\draw[-] (C1) to (E1);
		\draw[-] (C1) to (F1);
		\draw[-] (D1) to (G1);
		\draw[-] (E1) to (H1);
		\draw[-] (F1) to (I1);
		\draw[-] (F1) to (J1);
		
		\node (A2) at (9.3,0) [draw,fill=gray,rectangle]  {};
		\node (B2) at (9.3,-1.2) [draw,fill=gray]  {};
		\node (C2) at (10.6,-1.2) [draw,fill=gray]  {};
		\node (D2) at (8,-1.2) [draw,fill=gray]  {};
		\node (E2) at (11.1,-2.4) [draw,fill=gray]  {};
		\node (F2) at (10.1,-2.4) [draw,fill=gray]  {};
		\node (G2) at (8,-2.4) [draw,fill=gray]  {};
		\node (H2) at (11.3,-3.6) [draw,fill=gray]  {};
		\node (I2) at (9.6,-3.6) [draw,fill=gray]  {};
		\node (J2) at (10.6,-3.6) [draw,fill=gray]  {};
		\draw[-] (A2) to (B2);
		\draw[-] (A2) to (C2);
		\draw[-] (A2) to (D2);
		\draw[-] (C2) to (E2);
		\draw[-] (C2) to (F2);
		\draw[-] (D2) to (G2);
		\draw[-] (E2) to (H2);
		\draw[-] (F2) to (I2);
		\draw[-] (F2) to (J2);
	\end{tikzpicture} 
	\caption{A rooted labelled tree $H$ with root $3$ on the left--hand side together with the corresponding plane tree on the right--hand side. The tree in the middle reflects $H$ after ordering the children of each vertex according to their labels.}
	\label{fig:plane_tree}
\end{figure}

A \GW\ tree $\gwt{c}$ with offspring distribution $\poisson{c}$ can also be considered as a plane tree if we build $\gwt{c}$ via a branching process: We start by revealing the number of children of the root $v_0$ by drawing a random variable $X_0$ according to the $\poisson{c}$--distribution. Then we let $v_1, \ldots, v_{X_0}$ be the children of $v_0$ and determine the number of children of $v_1, \ldots, v_{X_0}$ one after another by drawing i.i.d. $\poisson{c}$--distributed random variables. Next, we sequentially choose the number of children of these new vertices and so on. Now it suffices to show the following modified version of \Cref{lem:local_er}. Roughly speaking, it is easier to show \Cref{lem:local_er2} than proving \Cref{lem:local_er} directly, because considering plane trees breaks some possible symmetries in a rooted unlabelled tree. 
\begin{lem}\label{lem:local_er2}
	Let $m=m(n)=\alpha n/2$ where $\alpha=\alpha(n)$ tends to a constant $c\geq 0$, $\planeTree$ be a plane tree, $\radius\in\N$, and $G=G(n,m)\ur\erClass(n,m)$. Then \whp
	\begin{align*}
		n^{-1}\cdot\Big|\setbuilderBig{v\in\vertexSet{G}}{\ball{\radius}{G}{v}=\planeTree}\Big|=\left(1+\smallo{1}\right)\prob{\ballP{\radius}{\gwt{c}}=\planeTree},
	\end{align*}
where the \GW\ tree $\gwt{c}$ is considered as a plane tree.
\end{lem}
\begin{proof}
	To simplify notation, we set for $v\in\vertexSet{G}$
\begin{align*}
X_v:=
\begin{cases}
	1  & \text{~~if~~} \ball{\radius}{G}{v}=\planeTree,\\
	0  & \text{~~otherwise~~}.
\end{cases}	
\end{align*}
We also define $X:=\sum_{v}X_v$ and $\gamma:=\prob{\ballP{\radius}{\gwt{c}}=\planeTree}$. Using these notations we can rewrite the statement as \whp\ $X=\left(1+\smallo{1}\right)\gamma n$. Let $v_1, \ldots, v_{\numberVertices{\planeTree}}$ be the vertices of $\planeTree$ in the order they are found by a breadth--first algorithm which starts at the root of $\planeTree$ and chooses the order of the children according to their order in $\planeTree$. If $\distance{\planeTree}{v_1}{v_{\numberVertices{\planeTree}}}>\radius$, then we obviously have $X=\gamma=0$. Thus, we can restrict our considerations to the case $\distance{\planeTree}{v_1}{v_{\numberVertices{\planeTree}}}\leq\radius$. Now let $k\in[{\numberVertices{\planeTree}}]$ be such that $\distance{\planeTree}{v_1}{v_i}<\radius$ for all $i\leq k$ and $\distance{\planeTree}{v_1}{v_i}=\radius$ if $i>k$. For each $i\in[k]$ we denote by $d_i$ the number of children of $v_i$ in $\planeTree$, i.e. $d_1=\degree{v_1}{\planeTree}$ and $d_i=\degree{v_i}{\planeTree}-1$ for $i\geq 2$, where $\degree{v_i}{\planeTree}$ is the degree of $v_i$ in $\planeTree$. By revealing the number of children of each vertex in the \GW\ tree $\gwt{c}$ sequentially, we obtain
\begin{align}\label{eq:16}
\gamma=\prod_{i=1}^{k}\frac{e^{-c}c^{d_i}}{d_i!}=e^{-ck}c^{\numberVertices{\planeTree}-1}\prod_{i=1}^{k}\frac{1}{d_i!},
\end{align}
where we used in the last equality that $\sum_{i=1}^{k}d_i=\numberVertices{\planeTree}-1$.

In the next step, we show \whp\ $X=\left(1+\smallo{1}\right)\gamma n$ by the first and second moment method. To that end, it suffices to prove $\prob{X_v=1}=\left(1+\smallo{1}\right)\gamma$ for each $v\in\vertexSet{G}$ and $\prob{X_v=X_w=1}=\left(1+\smallo{1}\right)\gamma^2$ for all $v\neq w$. For the former one we fix a vertex $v\in[\vertexSet{G}]$ and estimate the number of graphs $H\in \mathcal{G}(n,m)$ with $\ball{\radius}{H}{v}=\planeTree$. Each such graph can be constructed by first creating $\ball{\radius}{H}{v}$ and then choosing the edges outside $\ball{\radius}{H}{v}$. To obtain $\ball{\radius}{H}{v}$ we choose sets of neighbours one after another. More precisely, we start by picking a subset $S_1\subseteq[n]\setminus\left\{v\right\}$ of size $d_1$ for the neighbourhood of $v$. Then we choose $S_2\subseteq[n]\setminus\left(\left\{v\right\}\cup S_1\right)$ of size $d_2$ for the neighbours of the next vertex and so on. We note that given the set of neighbours for a vertex the ordering of them in the plane tree $\ball{\radius}{H}{v}$ is already determined by the labels of the neighbours. Hence, the number of different ways of choosing the sets $S_1, \ldots, S_k$ of neighbours and thereby building $\ball{\radius}{H}{v}$ such that $\ball{\radius}{H}{v}=\planeTree$ is
\begin{align}\label{eq:14}
\prod_{i=1}^{k} \binom{n-1-D_{i-1}}{d_i}=\left(1+\smallo{1}\right)n^{\numberVertices{\planeTree}-1}\prod_{i=1}^{k}\frac{1}{d_i!},
\end{align}
where $D_i:=\sum_{j=1}^{i}d_{j}$ and $D_k=\numberVertices{\planeTree}-1$. Now we assume that $B=\ball{\radius}{H}{v}$ is already built and we want to choose the set $E$ of edges outside $B$. Let $B'$ be the set of vertices $w\in\vertexSet{B}$ with $\distance{B}{v}{w}<\radius$. We note that $|B'|=k$ and the endpoints of the edges in $E$ lie in $[n]\setminus B'$ but not both endpoints of an edge can lie in $\vertexSet{B}\setminus B'$.  
Thus, the number of possible choices of $E$ is
\begin{align}\label{eq:15}
\binom{\binom{n-k}{2}-\binom{\numberVertices{\planeTree}-k}{2}}{m-\numberVertices{\planeTree}+1}=\left(1+\smallo{1}\right)e^{-ck}\left(c/n\right)^{\numberVertices{\planeTree}-1}\binom{\binom{n}{2}}{m},
\end{align}
where the equality follows by a standard asymptotic computation. Combining (\ref{eq:14}) and (\ref{eq:15}) yields
\begin{align}
\Big|\setbuilder{H\in\mathcal{G}(n,m)}{\ball{\radius}{H}{v}=\planeTree}\Big|=\left(1+\smallo{1}\right)e^{-ck}c^{\numberVertices{\planeTree}-1}\binom{\binom{n}{2}}{m}\prod_{i=1}^{k}\frac{1}{d_i!}.
\end{align}
Together with (\ref{eq:16}) this implies
\begin{align}\label{eq:17}
\prob{X_v=1}=\frac{\left|\setbuilder{H\in\mathcal{G}(n,m)}{\ball{\radius}{H}{v}=\planeTree}\right|}{\left|\mathcal{G}(n,m)\right|}=\left(1+\smallo{1}\right)\gamma.
\end{align}

In the next step, we show $\prob{X_v=X_w=1}=\left(1+\smallo{1}\right)\gamma^2$ for all $v\neq w$. To that end, we start by estimating the number of graphs $H\in\mathcal{G}(n,m)$ which satisfy $\ball{\radius}{H}{v}=\ball{\radius}{H}{w}=\planeTree$ and $\ball{\radius}{H}{v}\cap\ball{\radius}{H}{w}=\emptyset$, i.e. $\ball{\radius}{H}{v}$ and $\ball{\radius}{H}{w}$ share no vertex and considered as plane trees they coincide both with $\planeTree$. Each such graph can be constructed by choosing sequentially $\ball{\radius}{H}{v}$, $\ball{\radius}{H}{w}$, and then the remaining edges. Analogously to (\ref{eq:14}) and (\ref{eq:15}), the possible number of doing that is
\begin{align}\label{eq:18}
&\Big|\setbuilder{H\in\mathcal{G}(n,m)}{\ball{\radius}{H}{v}=\ball{\radius}{H}{w}=\planeTree, \ball{\radius}{H}{v}\cap\ball{\radius}{H}{w}=\emptyset}\Big|\nonumber
\\&=\prod_{i=1}^{k} \binom{n-2-D_{i-1}}{d_i}\cdot \prod_{i=1}^{k} \binom{n-1-\numberVertices{\planeTree}-D_{i-1}}{d_i}\cdot\binom{\binom{n-2k}{2}-2\binom{\numberVertices{\planeTree}-k}{2}}{m-2\numberVertices{\planeTree}+2}\nonumber
\\
&=\left(1+\smallo{1}\right)\gamma^2\binom{\binom{n}{2}}{m}.
\end{align}
Next, we prove that the number of graphs $H\in\mathcal{G}(n,m)$ with $\ball{\radius}{H}{v}\cap\ball{\radius}{H}{w}\neq\emptyset$ is \lq small\rq, which will imply together with (\ref{eq:18}) that $\prob{X_v=X_w=1}=\left(1+\smallo{1}\right)\gamma^2$. To that end, let $i\in\N$ be fixed. Each graph $H\in\mathcal{G}(n,m)$ with $\distance{H}{v}{w}=i$ can be constructed (not necessarily in a unique way) by choosing first a path of length $i$ from $v$ to $w$ and then picking the $m-i$ remaining edges. Hence, we obtain
\begin{align*}
\Big|\setbuilder{H\in\mathcal{G}(n,m)}{\distance{H}{v}{w}=i}\Big|\leq (n-2)(n-3)\ldots(n-i)\binom{\binom{n}{2}-i+1}{m-i}=\bigo{n^{-1}}\binom{\binom{n}{2}}{m}.
\end{align*}
In particular, this implies
\begin{align*}
	\Big|\setbuilder{H\in\mathcal{G}(n,m)}{\ball{\radius}{H}{v}\cap\ball{\radius}{H}{w}\neq\emptyset}\Big|=\bigo{n^{-1}}\binom{\binom{n}{2}}{m}.
\end{align*}
Combining that with (\ref{eq:18}) yields
\begin{align*}
\Big|\setbuilder{H\in\mathcal{G}(n,m)}{\ball{\radius}{H}{v}=\ball{\radius}{H}{w}=\planeTree}\Big|=\left(1+\smallo{1}\right)\gamma^2\binom{\binom{n}{2}}{m}.
\end{align*}
Finally, this implies
\begin{align}\label{eq:1}
\prob{X_v=X_w=1}=
\frac{\left|\setbuilder{H\in\mathcal{G}(n,m)}{\ball{\radius}{H}{v}=\ball{\radius}{H}{w}=\planeTree}\right|}{\left|\mathcal{G}(n,m)\right|}=\left(1+\smallo{1}\right)\gamma^2.
\end{align}
Thus, the statement follows by (\ref{eq:17}), (\ref{eq:1}), and the first and second moment method.
\end{proof}

\end{document}